\documentclass[11pt]{amsart}
\usepackage{amssymb
,amsthm
,amsmath
,amscd
,mathtools
,mathdots
}
\usepackage[all]{xy}
\usepackage[top=35truemm,bottom=35truemm,left=30truemm,right=30truemm]{geometry}
\usepackage[T3,T1]{fontenc}
\usepackage[usenames]{color}

\newcommand{\C}{\mathbb{C}}
\newcommand{\K}{\mathbb{K}}

\renewcommand{\AA}{\mathcal{A}}
\newcommand{\GG}{\mathcal{G}}
\newcommand{\HH}{\mathcal{H}}
\newcommand{\TT}{\mathcal{T}}
\newcommand{\UU}{\mathcal{U}}
\newcommand{\YY}{\mathcal{Y}}

\newcommand{\B}{\mathbf{B}}
\newcommand{\G}{\mathbf{G}}
\renewcommand{\H}{\mathbf{H}}
\newcommand{\T}{\mathbf{T}}
\newcommand{\U}{\mathbf{U}}

\newcommand{\GL}{\mathrm{GL}}
\newcommand{\SL}{\mathrm{SL}}
\renewcommand{\u}{\mathrm{U}}
\newcommand{\M}{\mathrm{M}}

\newcommand{\Cl}{\mathit{Cl}}
\newcommand{\WD}{\mathit{WD}}

\newcommand{\Irr}{\mathrm{Irr}}
\newcommand{\temp}{\mathrm{temp}}
\newcommand{\diag}{\mathrm{diag}}
\newcommand{\srs}{\mathrm{srs}}
\renewcommand{\ss}{\mathrm{ss}}
\newcommand{\st}{\mathrm{st}}
\newcommand{\Res}{\mathrm{Res}}
\newcommand{\Gal}{\mathrm{Gal}}
\newcommand{\vol}{\mathrm{vol}}
\newcommand{\id}{\mathrm{id}}
\newcommand{\tr}{\mathrm{tr}}
\newcommand{\Hom}{\mathrm{Hom}}
\newcommand{\Pic}{\mathrm{Pic}}

\newcommand{\iif}{&\quad&\text{if }}
\newcommand{\resp}{resp.~}
\renewcommand{\1}{\mathbf{1}}
\newcommand{\ep}{\varepsilon}
\newcommand{\bs}{\backslash}

\newcommand{\oo}{\mathfrak{o}}
\newcommand{\pp}{\mathfrak{p}}

\newcommand{\pair}[1]{\left\langle #1 \right\rangle}
\newcommand{\tl}[1]{\tilde{#1}}

\newtheorem{thm}{Theorem}[section]
\newtheorem{lem}[thm]{Lemma}
\newtheorem{prop}[thm]{Proposition}
\newtheorem{cor}[thm]{Corollary}
\newtheorem{defi}[thm]{Definition}
\newtheorem{rem}[thm]{Remark}

\title[Local newforms for odd unitary groups and Fundamental Lemma]
{Local newforms for generic representations of unramified odd unitary groups and Fundamental Lemma}
\author{
Hiraku Atobe, 
Masao Oi, \and
Seidai Yasuda
}
\date{}

\address{
Department of Mathematics, Hokkaido University,
Kita 10, Nishi 8, Kita-Ku, Sapporo, Hokkaido, 060-0810, Japan 
}
\address{
Department of Mathematics (Hakubi center), Kyoto University, 
Kitashirakawa, Oiwake- cho, Sakyo-ku, Kyoto 606-8502, Japan
}
\address{
Department of Mathematics, Hokkaido University,
Kita 10, Nishi 8, Kita-Ku, Sapporo, Hokkaido, 060-0810, Japan 
}
\email{
atobe@math.sci.hokudai.ac.jp
}
\email{
masaooi@math.kyoto-u.ac.jp
}
\email{
sese@math.sci.hokudai.ac.jp
}

\allowdisplaybreaks
\begin{document}
\maketitle

\begin{abstract}
In this paper, 
we establish the theory of local newforms 
for irreducible tempered generic representations of 
unramified odd unitary groups over a non-archimedean local field. 
For the proof, we prove an analogue of the fundamental lemma for our compact groups.
\end{abstract}

%\section{Introduction}
%\section{Introduction}
\section{Introduction}
For a connection between classical modular forms and automorphic representations, 
it is a significant problem in representation theory of $p$-adic groups
to establish the theory of \emph{local newforms}.
It was initiated by Casselman \cite{C} for $\GL_2$, 
and generalized by Jacquet--Piatetskii-Shapiro--Shalika \cite{JPSS} to the case for $\GL_N$ (see also \cite{J}).
For other groups, there are some progresses: 

\begin{itemize}
\item
Roberts--Schmidt \cite{RS1,RS2} established the theory of local newforms 
for $\mathrm{PGSp}_4$ and for $\widetilde{\SL}_2$, which is the double cover of $\SL_2$. 

\item
Okazaki \cite{Ok} generalized results in \cite{RS1} to $\mathrm{GSp}_4$. 

\item
Unramified $\mathrm{U}(1,1)$ and $\mathrm{U}(2,1)$ 
were treated by Lansky--Raghuram \cite{LR} and 
Miyauchi \cite{M1,M2,M3,M4}, respectively.

\item
Tsai \cite{Ts} proved Gross' conjecture \cite{G} on 
the local newforms for $\mathrm{SO}_{2n+1}$
for generic supercuspidal representations. 

\item
Recently, the first and third authors together with Kondo \cite{AKY}
extended results in \cite{JPSS} to all irreducible representations of $\GL_N$. 
\end{itemize}
\vskip 10pt

In this paper, we treat unramified unitary groups $\u_{2n+1}$ of odd variables. 
Let us describe our results more precisely. 
Let $E/F$ be an unramified quadratic extension 
of non-archimedean local fields of characteristic $0$ and 
of residue characteristic $p > 2$.
Denote by $\oo_E$ and $\pp_E$ the ring of integers of $E$ and its maximal ideal, respectively.
We consider a quasi-split unitary group
\[
H = \u_{2n+1} = \{g \in \GL_{2n+1}(E) \;|\; g J_{2n+1} {}^t\overline{g} = J_{2n+1}\}
\]
with 
\[
J_N = 
\begin{pmatrix}
&&&1\\
&&-1&\\
&\iddots&&\\
(-1)^{N-1}&&&
\end{pmatrix}.
\]
Set $\K_{0,H} = \GL_{2n+1}(\oo_E) \cap H$. 
For a positive integer $m$, we define 
\[
\K_{m,H} = 
\bordermatrix{
&n&1&n \cr
n&\oo_E&\oo_E&\pp_E^{-m}\cr
1&\pp_E^m&1+\pp_E^m&\oo_E\cr
n&\pp_E^m&\pp_E^m&\oo_E
} \cap H.
\]
\vskip 10pt

The following is the first main result. 
\begin{thm}\label{main1}
Let $\pi$ be an irreducible tempered representation of $H$. 
\begin{enumerate}
\item
If $\pi$ is not generic, then $\pi^{\K_{m,H}} = 0$ for all $m \geq 0$. 
\item
Suppose that $\pi$ is generic. 
We denote by $c(\phi)$ the conductor of the $L$-parameter $\phi$ of $\pi$ 
(see Section \ref{sec.mult1}). 
Then 
\[
\dim(\pi^{\K_{m,H}}) = 
\left\{
\begin{aligned}
&0 \iif m < c(\phi), \\
&1 \iif m = c(\phi).
\end{aligned}
\right. 
\]
Moreover, if $m > c(\phi)$, then $\pi^{\K_{m,H}} \not= 0$.
We call a nonzero element in $\pi^{\K_{c(\phi),H}}$ a \emph{local newform} of $\pi$. 
\end{enumerate}
\end{thm}

Theorem \ref{main1} (1) is an application of the \emph{local Gan--Gross--Prasad (GGP) conjecture} 
(\cite[Conjecture 17.3]{GGP}) established by Beuzart-Plessis \cite{BP1,BP2,BP3}. 
More precisely, 
if an irreducible tempered representation $\pi$ of $H$ 
satisfies $\pi^{\K_{m,H}} \not= 0$ for some $m \geq 0$, 
then the argument of a lemma of Gan--Savin (\cite[Lemma 12.5]{GS}) 
implies that 
there exists an irreducible unramified tempered representation $\pi'$ of $H' \cong \u_{2n}$, 
which we regard as a subgroup of $H$, 
such that $\Hom_{H'}(\pi \otimes \pi', \C) \not= 0$. 
In this situation, the local GGP conjecture tells us that $\pi$ is generic. 
For a detail, see Theorem \ref{generic} below.
\vskip 10pt

In all of the previous works listed above, 
the proofs of analogues of Theorem \ref{main1} (2) used \emph{Rankin--Selberg integrals}.
However, we do not consider them in this paper.
Instead of Rankin--Selberg integrals, 
we prove an analogue of the \emph{fundamental lemma} for our compact groups.
\vskip 10pt

The fundamental lemma is a core in the theory of \emph{endoscopy}.
Let $\tl{G} = G \rtimes \theta$ be a twisted space with an involution $\theta$ on a reductive group $G$, 
and let $H$ be an endoscopic group of $\tl{G}$. 
Suppose that both $G$ and $H$ are unramified. 
In this setting, the fundamental lemma asserts that 
the characteristic function $\1_{\K_H}$ is a \emph{transfer} of $\1_{\tl{\K}}$, 
where $\K \subset G$ and $\K_H \subset H$ are hyperspecial maximal compact subgroups, 
and we set $\tl\K = \K \rtimes \theta \subset \tl{G}$. 
For the notion of transfer, see Definition \ref{match} below. 
The fundamental lemma was established by Ng{\^o} \cite{Ngo} 
after Waldspurger's consecutive works \cite{W2,W3}. 
For some special cases, e.g., 
in Laumon--Ng{\^o} \cite{LN} and Kottwitz \cite{K-BC}, 
the fundamental lemma was earlier proven. 
\vskip 10pt

The \emph{transfer conjecture}, which is a consequence of the fundamental lemma (\cite{W1, W3}), 
says that for any $f \in C_c^\infty(\tl{G})$, 
there exists $f^H \in C_c^\infty(H)$ such that $f^H$ is a transfer of $f$. 
However, its proof is not constructive. 
It is of independent interest to find explicit pairs $(f,f^H)$ such that $f^H$ is a transfer of $f$. 
Beyond the characteristic functions of hyperspecial maximal compact subgroups, 
a few cases are known. 
For example: 

\begin{itemize}
\item
Clozel \cite{Cl} and Labesse \cite{L} (\resp Haines \cite{Haines}) 
established the base change fundamental lemma for spherical Hecke algebras
(\resp for the central elements in parahoric Hecke algebras). 

\item
Hales \cite{Hales} and Lemaire--M{\oe}glin--Waldspurger \cite{LMW}
extended the (usual) fundamental lemma to all elements in the spherical Hecke algebras.

\item
Kazhdan--Varshavsky \cite[Theorem 2.2.6]{KV} proved a formula for the transfer of Deligne--Lusztig functions 
when $\theta = \id$.
It is an answer to a conjecture of Kottwitz. 
Also, it implies that the fundamental lemma for Iwahori subgroups (\cite[Corollary 2.3.13]{KV}). 

\item
The fundamental lemma for principal congruence subgroups 
was shown by Ferrari \cite[Th{\'e}or{\`e}me 3.2.3]{F} when $\theta = \id$
(\resp by Ganapathy--Varma \cite[Lemma 8.2.2]{GV} when $G = \GL_N$ and $\theta \not= \id$).

\item
The second author also showed the fundamental lemma for 
the Moy--Prasad filtrations of unitary groups (\cite[Theorem 1.4]{Oi}).
\end{itemize}
We remark that in these papers, 
some assumption on the residue characteristic $p$ might be imposed.
\vskip 10pt

Now we state the second main result. 
Let $\K_0 = \GL_{2n+1}(\oo_E) \subset G = \GL_{2n+1}(E)$. 
For a positive integer $m > 0$, 
we define $\K_{m}$ by the compact subgroup of $G$ consisting of matrices $k$
of the form
\[
\bordermatrix{
&n&1&n \cr
n&\oo_E&\oo_E&\pp_E^{-m}\cr
1&\pp_E^m&1+\pp_E^m&\oo_E\cr
n&\pp_E^m&\pp_E^m&\oo_E
}
\]
with $\det(k) \in \oo_E^\times$.
Hence $\K_{m,H} = \K_m \cap H$.
Set $\tl\K_m = \K_m \rtimes \theta \subset \tl{G} = G \rtimes \theta$, 
where $\theta$ is an involution on $G$ such that $H$ is the centralizer of $\theta$ in $G$.

\begin{thm}[Theorem \ref{transfer}]\label{main2}
For $m > 0$, the function $\vol(\K_{m,H}; dh)^{-1}\1_{\K_{m,H}} \in C_c^\infty(H)$ 
is a transfer of $\vol(\K_{m}; dg)^{-1}\1_{\tl{\K}_m} \in C_c^\infty(\tl{G})$.
\end{thm}

By this theorem, 
one can transfer a result in \cite{JPSS} on the local newforms for $G = \GL_{2n+1}(E)$ 
to the one for $H = \u_{2n+1}$ via the local Langlands correspondence 
established by Mok \cite{Mok}. 
More precisely, one can prove the multiplicity one result in tempered $L$-packets (Theorem \ref{packet=1}).
Combining it with Theorem \ref{main1} (1), 
we obtain Theorem \ref{main1} (2).
\vskip 10pt

An advantage for using Theorem \ref{main2}
is that we do not need huge computations of elements in Hecke algebras 
as in \cite[Section 6]{RS1} and \cite[Chapter 8]{Ts}.
In fact, after reducing Theorem \ref{main2} 
to comparing topologically semisimple conjugacy classes (Lemma \ref{descent}), 
we only need a result of Kottwitz \cite[Proposition 7.1]{K-stable} and its twisted analogue (Proposition \ref{71})
together with easy linear algebraic calculations (Proposition \ref{eigen=1}). 
On the other hand, 
the relation between Rankin--Selberg integrals and newforms is not clear in this paper. 
\vskip 10pt

This paper is organized as follows. 
In Section \ref{sec.endoscopy}, we review the theory of endoscopy. 
Especially, we state our second main result as Theorem \ref{transfer}, 
and reduce it to comparing certain conjugacy classes in Lemma \ref{descent}. 
This comparison is investigated in Section \ref{sec.comparison}. 
Finally, we prove Theorem \ref{main1} in Section \ref{sec.newforms}. 
In Appendix \ref{appA}, we generalize Hilbert's Theorem 90 by using the faithfully flat descent.
\vskip 10pt

%\subsection*{Acknowledgement}
\subsection*{Acknowledgement}
The first author was supported by JSPS KAKENHI Grant Number 19K14494.
The third author was supported by JSPS KAKENHI Grant Number 21H00969.

%\subsection*{Notation}
\subsection*{Notation}
Let $E/F$ be an unramified quadratic extension 
of non-archimedean local fields of characteristic $0$ and of residue characteristic $p > 2$.
The non-trivial element in $\Gal(E/F)$ is denoted by $x \mapsto \overline{x}$.
Set $\oo_E$ (\resp $\oo_F$) to be
the ring of integers of $E$ (\resp $F$), and 
$\pp_E$ (\resp $\pp_F$) to be its maximal ideal. 
Fix a uniformizer $\varpi$ of $F$, which is also a uniformizer of $E$. 
Write $q = |\oo_F/\pp_F|$ so that $q^2 = |\oo_E/\pp_E|$.
\par

The identity matrix of size $N$ is denoted by $I_N$.
Set
\[
J_N = 
\begin{pmatrix}
&&&1\\
&&-1&\\
&\iddots&&\\
(-1)^{N-1}&&&
\end{pmatrix}.
\]
When $N=2n$, we also put
\[
J'_{2n} = 
\begin{pmatrix}
0 & J_{n} \\ 
{}^tJ_n & 0
\end{pmatrix}
=
\left(
\begin{array}{ccc|ccc}
&&&&&1\\
&&&&\iddots&\\
&&&(-1)^{n-1}&&\\
\hline
&&(-1)^{n-1}&&&\\
&\iddots&&&&\\
1&&&&&
\end{array}
\right).
\]
\par

For an algebraic variety $\mathbf{X}$ over $F$, 
we denote the set of its $F$-valued points by $X = \mathbf{X}(F)$.
\par

%\section{Matching of orbital integrals and descent to centralizers}
%\section{Matching of orbital integrals and descent to centralizers}
\section{Matching of orbital integrals and descent to centralizers}\label{sec.endoscopy}
In this section, we review several notions concerning transfers, 
and we reduce our problem to a correspondence of certain semisimple conjugacy classes. 

%\subsection{Groups}
\subsection{Groups}\label{sec.groups}
Let $\G = \Res_{E/F}(\GL_{N,E})$ be the Weil restriction of $\GL_{N}$ over $E$.
In particular, $G = \G(F) = \GL_N(E)$. 
Define an involution $\theta$ on $\G$ by 
\[
\theta(x) = J_N{}^t\overline{x}^{-1}J_N^{-1}. 
\]
Let $\tl{\G} = \G \rtimes \theta$ be the non-neutral component of the group $\G \rtimes \pair{\theta}$. 
Then $\tl{\G}$ is a twisted space over $\G$. 
Let $\H = \G_{\theta} = \{ h \in \G \;|\; \theta(h) = h\}$ be a unitary group. 
It is an endoscopic group for $\tl{\G}$. 
\par

For $\gamma \in \H$, write $\H_\gamma$ for its centralizer. 
Similarly, for $\tl{\delta} = \delta \rtimes \theta \in \tl{\G}$, 
one can consider its conjugate 
\[
g \tl{\delta} g^{-1} = g(\delta \rtimes \theta) g^{-1} = (g \delta \theta(g)^{-1}) \rtimes \theta
\]
for $g \in \G$.
Set 
\[
\G_{\tl{\delta}} = \{g \in \G \;|\; g \tl{\delta} g^{-1} = \tl{\delta}\}
\]
to be the centralizer of $\tl{\delta}$.
We sometimes write $N(\delta) = \delta \theta(\delta) = \tl\delta^2$.
Note that $\G_{\tl\delta} \subset \G_{N(\delta)}$.
\par

Assume until the end of this subsection that $N = 2n+1$ is odd. 
Let $\G' = \Res_{E/F}(\GL_{2n,E})$ with an inclusion 
\[
\iota \colon \G' \hookrightarrow \G,\;
\begin{pmatrix}
A & B \\ C & D
\end{pmatrix}
\mapsto 
\begin{pmatrix}
A &0& B\\
0&1&0 \\
C&0&D
\end{pmatrix}
\]
for matrices $A,B,C,D$ of size $n \times n$.
Note that if we set $\H' = \iota^{-1}(\H)$, then $\H' = \G'_{\theta'}$
with
\[
\theta'(x) = J'_{2n}{}^t\overline{x}^{-1}{J'}_{2n}^{-1}. 
\]
If we fix $\epsilon \in \oo_E^\times$ with $\overline{\epsilon} = -\epsilon$, 
and if we set 
$\kappa = \diag(\underbrace{\epsilon, \dots, \epsilon}_n, \underbrace{1, \dots, 1}_n) \in \G'$, 
then 
\[
\epsilon J'_{2n} = \kappa J_{2n} {}^t\overline{\kappa}
\]
so that $\theta'(x) = \kappa \cdot \theta(\kappa^{-1}x\kappa) \cdot \kappa^{-1}$.
In particular, we have $\G'_{\theta'} = \kappa \G'_{\theta} \kappa^{-1}$. 
\par

Set $\K_0 = \GL_{2n+1}(\oo_E)$, 
which is a hyperspecial maximal compact subgroup of $G$.
For a fixed positive integer $m > 0$, 
we define $\K_{m}$ by the compact open subgroup of $G$ consisting of matrices $k$
of the form
\[
\bordermatrix{
&n&1&n \cr
n&\oo_E&\oo_E&\pp_E^{-m}\cr
1&\pp_E^m&1+\pp_E^m&\oo_E\cr
n&\pp_E^m&\pp_E^m&\oo_E
}
\]
with $\det(k) \in \oo_E^\times$.
Note that $\theta(\K_m) = \K_m$.
Moreover, $\iota^{-1}(\K_m)$ (\resp $\iota^{-1}(\K_m) \cap G'_{\theta'}$) 
is a hyperspecial maximal compact subgroup of $G'$ (\resp $G'_{\theta'}$).
Finally, we define $\K_{m,H} = \K_m \cap H$ for $m \geq 0$, which is a compact open subgroup of $H$. 

%\subsection{Norm correspondence}
\subsection{Norm correspondence}
For a subgroup $J \subset \G(\overline{F})$ (\resp $J_H \subset \H(\overline{F})$), 
we say that elements $\tl{\delta},\tl{\delta}' \in \tl\G(\overline{F})$ 
(\resp $\gamma, \gamma' \in \H(\overline{F})$)
are \emph{$J$-conjugate} (\resp \emph{$J_H$-conjugate}) 
if there exists $g \in J$ (\resp $h \in J_H$) 
such that $\tl\delta' = g \tl\delta g^{-1}$ (\resp $\gamma' = h \gamma h^{-1}$). 

\begin{rem}
For a general connected reductive group $\H$ over $F$, 
the $\H(\overline{F})$-conjugacy 
might be different from the so-called \emph{stably conjugacy}, 
which is a bit more complicated. 
Although they are the same in our setting, 
we do not use the terminology of the stably conjugacy in this paper. 
\end{rem}

Let $\Cl_{\ss}(\tl\G)$ (\resp $\Cl_{\ss}(\H)$) be the set of semisimple $\G(\overline{F})$-conjugacy classes 
in $\tl{\G}(\overline{F})$ (\resp $\H(\overline{F})$-conjugacy classes in $\H(\overline{F})$).
Here, see \cite[Section 3.2]{KS} for the definition of the semisimplicity for elements of $\tl{\G}$.
As in \cite[Theorem 3.3.A]{KS}, 
one can define a map
\[
\AA \colon \Cl_\ss(\H) \rightarrow \Cl_{\ss}(\tl\G)
\]
as follows. 
Let $[\gamma] \in \Cl_\ss(\H)$ and $[\tl\delta] \in \Cl_{\ss}(\tl\G)$. 
As representatives, 
we may take diagonal elements $\gamma = \diag(\zeta_1,\dots,\zeta_N)$
and $\tl\delta = \diag(x_1, \dots, x_N) \rtimes \theta$ 
with $\zeta_1,\dots,\zeta_N, x_1,\dots,x_N \in E \otimes_F \overline{F}$.
Then $\AA([\gamma]) = [\delta]$ if and only if 
\[
\{\zeta_1,\dots,\zeta_N\} = \left\{ x_1/\overline{x_N}, \dots, x_N/\overline{x_1} \right\}
\]
as multi-sets.
Note that $\AA$ is bijective in our setting (see \cite[Section 4.2]{Oi1}).
For semisimple elements $\gamma \in \H$ and $\tl\delta \in \tl\G$, 
we call $\gamma$ is a \emph{norm} of $\tl\delta$ if $\AA([\gamma]) = [\tl\delta]$.
In other words, $\gamma$ is a norm of $\tl\delta = \delta \rtimes \theta$ if and only if 
$\gamma$ is $\G(\overline{F})$-conjugate to $N(\delta)$. 
\par

\begin{defi}
\begin{enumerate}
\item
An element $\tl{\delta} = \delta \rtimes \theta \in \tl{\G}$ is
\emph{strongly regular semisimple}
if its centralizer $\G_{\tl\delta}$ is abelian.
Write $\tl{\G}^{\srs}$ for the set of strongly regular semisimple elements in $\tl{\G}$. 

\item
An element $\gamma \in \H$ is
\emph{strongly regular semisimple}
if its centralizer $\H_\gamma$ is a maximal torus of $\H$.
Write $\H^{\srs}$ for the set of strongly regular semisimple elements in $\H$. 

\item
An element $\gamma \in \H^\srs$ is 
\emph{strongly $\G$-regular semisimple}
if there is $\tl\delta \in \tl{\G}^\srs$ such that $\gamma$ is a norm of $\tl\delta$.
\end{enumerate}
\end{defi}

The rationality of semisimple conjugacy classes via 
the norm correspondence is a delicate problem in general. 
However, in our setting, we can prove the following.

\begin{prop}\label{90}
\begin{enumerate}
\item
For every semisimple element $\tl\delta \in \tl{G}$, 
there exists a semisimple element $\gamma \in H$ 
such that $\gamma$ is a norm of $\tl\delta$. 
Moreover, $\G_{\tl\delta}(\overline{F}) \cong \H_\gamma(\overline{F})$.

\item
For every semisimple element $\gamma \in H$, 
there exists a semisimple element $\tl\delta \in \tl{G}$ such that 
\begin{itemize}
\item
$\gamma$ is a norm of $\tl\delta$; and 
\item
$\G_{\tl\delta} = \H_\gamma$ as algebraic subgroups of $\G$. 
\end{itemize}
In particular, every $\gamma \in H^\srs$ is strongly $\G$-regular semisimple.
\end{enumerate}
\end{prop}
\begin{proof}
We prove (1). 
Let $\tl\delta = \delta \rtimes \theta \in \tl{G}$ be a semisimple element. 
We regard $N(\delta) = \delta \theta(\delta)$ as an element in $\H(E)$. 
Since $\theta(N(\delta)) = \delta^{-1} N(\delta) \delta$, 
the $\H(\overline{F})$-conjugacy class of $N(\delta)$ in $\H(\overline{F})$ 
is $\Gal(\overline{F}/F)$-invariant. 
Since $\H$ is quasi-split and 
its derived group is simply connected, 
by \cite[Theorem 4.1]{K-rational}, 
there exist $\gamma \in H = \H(F)$ and $h \in \H(\overline{F}) = \GL_N(\overline{F})$ 
such that $\gamma = h N(\delta) h^{-1}$. 
Especially, $\gamma$ is semisimple and is a norm of $\tl\delta$. 
Since $E \otimes_F \overline{F} \cong \overline{F} \times \overline{F}$
with $\Gal(E/F)$ acting on the right hand side by switching the two factors, 
we have 
$\G(\overline{F}) \cong \GL_N(\overline{F}) \times \GL_N(\overline{F})$
with $\theta$ acting on the right hand side by 
$\theta(g_1,g_2) = (J_N{}^tg_2^{-1}J_N^{-1}, J_N{}^tg_1^{-1}J_N^{-1})$. 
Hence
\begin{align*}
\G_{\tl\delta}(\overline{F}) &\cong
\left\{ 
(g_1,g_2) \in \GL_N(\overline{F}) \times \GL_N(\overline{F}) \;\middle|\;
\delta J_N{}^tg_2^{-1}J_N^{-1} = g_1 \delta,\; 
\overline{\delta}J_N{}^tg_1^{-1}J_N^{-1} = g_2 \overline{\delta}
\right\}
\\&\cong 
\left\{ 
(g_1,g_2) \in \GL_N(\overline{F}) \times \GL_N(\overline{F}) \;\middle|\;
g_2 = \overline{\delta}J_N{}^tg_1^{-1}J_N^{-1}\overline{\delta}^{-1},\; 
N(\delta)g_1N(\delta)^{-1} = g_1
\right\}
\\&\cong 
\left\{ g_1 \in \GL_{N}(\overline{F}) \;\middle|\; N(\delta)g_1N(\delta)^{-1} = g_1 \right\}
\\&\cong 
\left\{ h_1 \in \GL_{N}(\overline{F}) \;\middle|\; \gamma h_1 \gamma^{-1} = h_1 \right\}
= \H_\gamma(\overline{F}).
\end{align*}
\par

Next, we show (2). 
Let $\gamma \in H$ be a semisimple element.
Since $[E^\times : F^\times] = \infty$,
there exists $\alpha \in E^\times$ 
such that $\delta = \alpha I_N + \overline{\alpha}\gamma$ is an invertible matrix so that $\delta \in G$.
Since $\gamma$ is semisimple in $\H$, 
we see that $\tl\delta = \delta \rtimes \theta$ is semisimple in $\tl{\G}$. 
We claim that $\tl\delta$ satisfies the desired conditions.
\par

Remark that the map $x \mapsto \theta(x)^{-1} = J_N{}^t\overline{x}J_N^{-1}$
can be extended to an anti-$E$-linear endomorphism of $\M_N(E)$.
In particular, $\theta(\delta)^{-1} = \overline{\alpha} I_N + \alpha \gamma^{-1}$
since $\theta(\gamma) = \gamma$. 
Hence $\gamma \theta(\delta)^{-1} = \delta$ so that $\gamma = \delta \theta(\delta)$. 
This means that $\gamma$ is a norm of $\tl\delta$. 
\par

Note that $\tl\delta^2 = \delta \theta(\delta) = \gamma$.
Therefore, both $\G_{\tl\delta}$ and $\H_\gamma$
are subgroups of the centralizer $\G_{\gamma}$ of $\gamma$ in $\G$.
Moreover, 
for $g \in \G_\gamma$, 
since $g \delta = \delta g$ by the construction of $\delta$, 
we have
\begin{align*}
g \in \G_{\tl\delta} 
&\iff 
g(\delta \rtimes \theta) = (\delta \rtimes \theta)g
\\&\iff
g\delta \rtimes \theta = \delta \theta(g) \rtimes \theta
\\&\iff 
g = \theta(g) 
\iff g \in \H_\gamma.
\end{align*}
Hence we conclude that $\G_{\tl\delta} = \H_{\gamma}$.
\end{proof}

%\subsection{Matching of orbital integrals}
\subsection{Matching of orbital integrals}
For $\tl{\delta} \in \tl{G}^\srs$ and $f \in C_c^\infty(\tl{G})$, 
we define the \emph{orbital integral} of $f$ at $\tl{\delta}$ by 
\[
O_{\tl{\delta}}(f) = \int_{G/G_{\tl\delta}} f(g \tl{\delta} g^{-1}) d\dot{g}.
\] 
Similarly, 
for $\gamma \in H^\srs$ and $f^H \in C_c^\infty(H)$, 
the \emph{orbital integral} of $f^H$ at $\gamma$ is defined by
\[
O_{\gamma}(f^H) = \int_{H/H_\gamma} f^H(h \gamma h^{-1}) d\dot{h}.
\]
Here, $d\dot{g}$ (\resp $d\dot{h}$) 
is a left $G$-invariant (\resp $H$-invariant) measure on $G/G_{\tl\delta}$ (\resp on $H/H_\gamma$)
induced by Haar measures on $G$ and $G_{\tl\delta}$ (\resp $H$ and $H_\gamma$).
Finally, we define the \emph{stable orbital integral} of $f^H$ at $\gamma$ by 
\[
SO_\gamma(f^H) = \sum_{\gamma' \sim_\st \gamma / \sim} O_{\gamma'}(f^H), 
\]
where 
the sum is over the set of $H$-conjugacy classes within the $\H(\overline{F})$-conjugacy class of $\gamma$.
\par

As explained in \cite[Section 4.1]{Oi1}, 
 $\H$ is an endoscopic group for $\tl\G$
(with respect to the standard base change $L$-embedding from ${}^L\H$ to ${}^L\G$).
Fix a $\theta$-stable $F$-splitting of $G$. 
Let $\Delta(\gamma, \tl\delta)$ be the Kottwitz--Shelstad transfer factor 
with respect to $(\tl{G}, H)$ (see \cite[Sections 4, 5]{KS}). 
By \cite[Proposition 4.3]{Oi1}, we know that $\Delta(\gamma, \tl\delta) = 1$
for $\gamma \in H^\srs$ and $\tl\delta \in \tl{G}^\srs$ such that $\gamma$ is a norm of $\tl\delta$.

\begin{defi}\label{match}
We say that $f \in C_c^\infty(\tl{G})$ and $f^H \in C_c^\infty(H)$
have \emph{matching orbital integrals} 
if, for every strongly $\G$-regular semisimple element $\gamma \in H^\srs$, 
the identity
\[
SO_\gamma(f^H) = \sum_{\tl{\delta} \leftrightarrow \gamma/\sim} \Delta(\gamma, \tl\delta) O_{\tl{\delta}}(f)
= \sum_{\tl{\delta} \leftrightarrow \gamma/\sim} O_{\tl{\delta}}(f)
\]
holds, 
where the sum is taken over the set of $G$-conjugacy classes of $\tl\delta \in \tl{G}^\srs$ 
such that $\gamma$ is a norm of $\tl\delta$.
Here, we choose measures as in the manner of \cite[Section 5.5]{KS}. 
In this situation, we say that $f^H$ is a \emph{transfer} of $f$.
\end{defi}

When $N$ is odd, 
we have defined compact open subgroups $\K_m \subset G$ and $\K_{m,H} \subset H$ 
in Section \ref{sec.groups}. 
In this and next sections, we will prove the following. 

\begin{thm}\label{transfer}
Suppose that $N$ is odd. 
For $m > 0$, 
the function $\vol(\K_{m,H}; dh)^{-1}\1_{\K_{m,H}} \in C_c^\infty(H)$ 
is a transfer of $\vol(\K_{m}; dg)^{-1}\1_{\tl{\K}_m} \in C_c^\infty(\tl{G})$.
\end{thm}

%\subsection{Descent}
\subsection{Descent}\label{sec.descent}
To compare orbital integrals, we introduce the following notions. 

\begin{defi}
\begin{enumerate}
\item
An element $u \in G$ (\resp $v \in H$) is called \emph{topologically unipotent}
if $\lim_{n \to \infty} u^{p^n} = 1$ (\resp $\lim_{n \to \infty} v^{p^n} = 1$).

\item
An element $s \rtimes \theta \in \tl{G}$ (\resp $t \in H$) is called \emph{topologically semisimple}
if $s \rtimes \theta$ (\resp $t$) is of finite order prime to $p$. 
Note that $s \rtimes \theta \in \tl{G}$ is topologically semisimple 
if and only if $N(s) = s\theta(s) \in G$ is topologically semisimple, i.e., 
$N(s)$ is of finite order prime to $p$.

\item
A \emph{topological Jordan decomposition (TJD)} 
of $\tl\delta = \delta \rtimes \theta \in \tl{G}$ (\resp $\gamma \in H$)
is a decomposition 
$\tl{\delta} = \tl{s}u = u \tl{s}$ (\resp $\gamma = tv = vt$)
with topologically semisimple $\tl{s} = s \rtimes \theta \in \tl{G}$ (\resp $t \in H$) 
and topologically unipotent $u \in G$ (\resp $v \in H$). 
In this situation, we write $\tl{s} = \tl\delta_s$ and $u = \tl\delta_u$ 
(\resp $t = \gamma_s$ and $v = \gamma_u$).
\end{enumerate}
\end{defi}

By \cite{S}, the TJD of $\tl\delta$ uniquely exists 
if $\tl{\delta}$ is in a compact subgroup of $G \rtimes \pair{\theta}$, 
or equivalently, $N(\delta) = \tl\delta^2$ is in a compact subgroup of $G$.
This is similar for the TJD of $\gamma$. 
Moreover, in this case, $\tl{\delta}_s$ and $\tl\delta_u$ (\resp $\gamma_s$ and $\gamma_u$) belong to
the closure of the subgroup generated by $\tl{\delta}$ (\resp $\gamma$).
\par

The following is the first step of the proof of Theorem \ref{transfer}. 
\begin{lem}\label{descent}
Fix $\gamma \in H^\srs$.
Suppose that $\gamma$ has a TJD $\gamma = \gamma_s \gamma_u$.
\begin{enumerate}
\item
For a compact open subgroup $K \subset G$ satisfying $\theta(K) = K$, 
we have 
\[
\sum_{\tl\delta \leftrightarrow \gamma/\sim}O_{\tl{\delta}}(\vol(K; dg)^{-1}\1_{\tl{K}}) 
= \sum_{\tl{s} \in I_\gamma(K)} SO_{u}(\vol(K \cap G_{\tl{s}}; dg_{\tl{s}})^{-1}\1_{K \cap G_{\tl{s}}}),
\]
where
$I_\gamma(K)$ is a set of representatives $\tl{s}$ of the $K$-conjugacy classes
satisfying the following conditions: 
\begin{itemize}
\item
$\tl{s} \in \tl{K}$; 
\item
$\tl{s} = s \rtimes \theta$ is topologically semisimple; 
\item
$\gamma_s$ is a norm of $\tl{s}$; 
\item
there exists $u \in K \cap G_{\tl{s}}$ such that $\gamma$ is a norm of $\tilde{s}u$. 
\end{itemize}

\item
For a compact open subgroup $K_H \subset H$, 
we have
\[
SO_\gamma(\vol(K_{H}; dh)^{-1}\1_{K_H}) 
= \sum_{t \in J_\gamma(K_H)} SO_{v}(\vol(K_{H} \cap H_t; dh_t)^{-1}\1_{K_H \cap H_{t}}), 
\]
where
$J_\gamma(K_H)$ is a set of representatives $t$ of 
the $K_H$-conjugacy classes satisfying the following conditions: 
\begin{itemize}
\item
$t \in K_H$; 
\item
$t$ is topologically semisimple; 
\item
$t$ is $\H(\overline{F})$-conjugate to $\gamma_s$; 
\item
there exists $v \in K_H \cap H_t$ such that 
$tv$ is $\H(\overline{F})$-conjugate to $\gamma$. 
\end{itemize}

\end{enumerate}
Here, the right hand sides are sums of stable orbital integrals for $G_{\tl{s}}$ and $H_{t}$, 
which are independent of the choices of $u \in G_{\tl{s}}$ and $v \in H_t$.
\end{lem}
\begin{proof}
First, we prove (2). 
The proof is the same as the one of \cite[Lemma 6.3.1]{KV}. 
Since $\vol(K_{H}; dh)^{-1}\1_{K_H}$ is bi-$K_H$-invariant, 
we have
\begin{align*}
&SO_{\gamma}(\vol(K_{H}; dh)^{-1}\1_{K_H}) 
\\&= \sum_{\gamma' \sim_\st \gamma /\sim} 
\int_{H/H_{\gamma'}} \vol(K_{H}; dh)^{-1}\1_{K_H}(h \gamma' h^{-1}) d\dot{h}
\\&= \sum_{\gamma' \sim_\st \gamma /\sim}
\int_{K_H \bs H/H_{\gamma'}} 
\int_{K_H/ K_H \cap H_{h \gamma' h^{-1}}}
\vol(K_{H}; dh)^{-1}\1_{K_H}(kh \gamma' h^{-1}k^{-1}) d\dot{k}d\dot{h}
\\&= 
\sum_{\gamma''} \vol(K_H \cap H_{\gamma''}; dh_{\gamma''})^{-1} \1_{K_H}(\gamma''), 
\end{align*}
where $\gamma'' \in K_H$ runs over a set of representatives of 
the $K_H$-conjugacy classes which are $\H(\overline{F})$-conjugate to $\gamma$.
Here, in the last equation, we set $\gamma'' = h \gamma' h^{-1}$, 
and we note that 
\[
\vol(K_H/K_H \cap H_{\gamma''}; d\dot{k}) 
= 
\frac{\vol(K_H; dh)}{\vol(K_H \cap H_{\gamma''}; dh_{\gamma''})}. 
\]
By considering the TJD of $\gamma''$, 
the sum $\sum_{\gamma''}$ above can be replaced with 
$\sum_{t \in J_\gamma(K_H)}\sum_{\gamma''}$, where 
\begin{itemize}
\item
$\gamma'' \in K_H$ runs over a set of representatives of 
the $K_H$-conjugacy classes which are $\H(\overline{F})$-conjugate to $\gamma$
such that $\gamma''_s = t$. 
\end{itemize} 
For each $t \in J_\gamma(K_H)$, we fix $\gamma''$ as above. 
If we write $\gamma'' = tv$ for its TJD, 
we can rewrite 
\[
SO_{\gamma}(\vol(K_{H}; dh)^{-1}\1_{K_H}) 
= 
\sum_{t \in J_\gamma(K_H)}\sum_{v'}
\vol(K_H \cap H_{tv'}; dh_{tv'})^{-1} \1_{K_H}(tv'),
\]
where 
$v' \in K_H \cap H_t$ runs over 
a set of representatives of the $(K_H \cap H_t)$-conjugacy classes 
which are $\H_t(\overline{F})$-conjugate to $v$.
\par

By a similar computation as above, 
since $H_t \cap H_{v'} = H_{tv'}$, 
we have
\begin{align*}
&SO_{v}(\vol(K_{H} \cap H_t; dh_t)^{-1}\1_{K_H \cap H_t}) 
\\&= \sum_{v' \sim_\st v /\sim}
\int_{H_t/H_{tv'}} \vol(K_{H} \cap H_t; dh_t)^{-1}\1_{K_H \cap H_t}(hv'h^{-1})d\dot{h}
\\&= 
\sum_{v'} \vol(K_H \cap H_{tv'}; dh_{tv'})^{-1}\1_{K_H \cap H_t}(v'), 
\end{align*}
where in the last sum, 
$v' \in K_H \cap H_t$ runs over a set of representatives of 
the $(K_H \cap H_t)$-conjugacy classes which are $\H_t(\overline{F})$-conjugate to $v$.
Namely, the sum $\sum_{v'}$ is the same as in the expression of 
$SO_{\gamma}(\vol(K_{H}; dh)^{-1}\1_{K_H})$ above.
By noting that $\1_{K_H}(tv') = \1_{K_H \cap H_t}(v') = 1$, 
we conclude that 
\[
SO_{\gamma}(\vol(K_{H}; dh)^{-1}\1_{K_H}) 
= \sum_{t \in J_\gamma(K_H)}SO_{v}(\vol(K_{H} \cap H_t; dh_t)^{-1}\1_{K_H \cap H_t}). 
\]
\par

The proof of (1) is essentially the same as the one of (2).
Since $\theta(K) = K$, 
we see that $\vol(K; dg)^{-1}\1_{\tl{K}}$ is bi-$K$-invariant. 
Hence we have
\begin{align*}
\sum_{\tl\delta \leftrightarrow \gamma/\sim}
&O_{\tl{\delta}}(\vol(K; dg)^{-1}\1_{\tl{K}}) 
\\&= 
\sum_{\tl\delta \leftrightarrow \gamma/\sim}
\int_{G/G_{\tl\delta}} \vol(K; dg)^{-1}\1_{\tl{K}}(g \tl{\delta} g^{-1}) d\dot{g}
\\&=
\sum_{\tl\delta \leftrightarrow \gamma/\sim}
\int_{K \bs G/G_{\tl\delta}} 
\int_{K/ K \cap G_{g \tl{\delta} g^{-1}}}
\vol(K; dg)^{-1}\1_{\tl{K}}(kg \tl{\delta} g^{-1}k^{-1}) d\dot{k}d\dot{g}
\\&= 
\sum_{\tl{\delta}'} \vol(K \cap G_{\tl\delta'}; dg_{\tl\delta'})^{-1} \1_{\tl{K}}(\tl\delta'), 
\end{align*}
where $\tl\delta' \in \tl{K}$ runs over a set of representatives of the $K$-conjugacy classes 
such that $\gamma$ is a norm of $\tl{\delta}'$. 
By considering the TJD of $\tl\delta'$, 
the sum $\sum_{\tl\delta'}$ above can be replaced with 
$\sum_{\tl{s} \in I_\gamma(K)}\sum_{\tl\delta'}$, where 
\begin{itemize}
\item
$\tl\delta' \in \tl{K}$ runs over a set of representatives of the $K$-conjugacy classes 
such that $\gamma$ is a norm of $\tl{\delta}'$, 
and such that $\tl\delta'_s = \tl{s}$. 
\end{itemize} 
For each $\tl{s} \in I_\gamma(K)$, we fix $\tl{\delta}'$ as above. 
If we write $\tl{\delta}' = \tl{s}u$ for its TJD, 
we can rewrite 
\[
\sum_{\tl\delta \leftrightarrow \gamma/\sim}
O_{\tl{\delta}}(\vol(K; dg)^{-1}\1_{\tl{K}}) 
= 
\sum_{\tl{s} \in I_\gamma(K)}\sum_{u'}
\vol(K \cap G_{\tl{s}u'}; dg_{\tl{s}u'})^{-1} \1_{\tl{K}}(\tl{s}u'),
\]
where 
$u' \in K \cap G_{\tl{s}'}$ runs over 
a set of representatives of the $(K \cap G_{\tl{s}'})$-conjugacy classes 
which are $\G_{\tl{s}'}(\overline{F})$-conjugate to $u$.
\par

By the same computation as in the proof of (2), 
since $(G_{\tl{s}})_{u'} = G_{\tl{s}u'}$, 
we have
\begin{align*}
&SO_{u}(\vol(K \cap G_{\tl{s}}; dg_{\tl{s}})^{-1}\1_{K \cap G_{\tl{s}}}) 
\\&= \sum_{u' \sim_\st u /\sim}
\int_{G_{\tl{s}}/(G_{\tl{s}})_{u'}} \vol(K \cap G_{\tl{s}}; dg_{\tl{s}})^{-1}\1_{K \cap G_{\tl{s}}}(gu'g^{-1})d\dot{g}
\\&= 
\sum_{u'} \vol(K \cap G_{\tl{s}u'}; dg_{\tl{s}u'})^{-1}\1_{K \cap G_{\tl{s}}}(u'), 
\end{align*}
where $\sum_{u'}$ is the same as in the expression of 
$\sum_{\tl\delta \leftrightarrow \gamma/\sim}
O_{\tl{\delta}}(\vol(K; dg)^{-1}\1_{\tl{K}})$ above. 
By noting that $\1_{\tl{K}}(\tl{s}u') = \1_{K \cap G_{\tl{s}}}(u') = 1$, 
we conclude that 
\[
\sum_{\tl\delta \leftrightarrow \gamma/\sim}
O_{\tl{\delta}}(\vol(K; dg)^{-1}\1_{\tl{K}}) 
= 
\sum_{\tl{s} \in I_\gamma(K)}SO_{u}(\vol(K \cap G_{\tl{s}}; dg_{\tl{s}})^{-1}\1_{K \cap G_{\tl{s}}}). 
\]
This completes the proof.
\end{proof}

We will prove the following key proposition in the next section. 
\begin{prop}\label{bij}
Suppose that $N$ is odd. 
For $m > 0$, 
set $\tl\K_m = \K_m \rtimes \theta \subset \tl{G}$ and $\K_{m,H} = \K_m \cap H$. 
Then for $\gamma \in H^\srs$, 
there exists a bijection 
\[
J_{\gamma}(\K_{m,H}) \longleftrightarrow I_{\gamma}(\K_m),\;
t \leftrightarrow \tl{s}.
\]
Moreover, by replacing $I_\gamma(\K_m)$ if necessary, 
we can assume that
\begin{itemize}
\item
$t$ is a norm of $\tl{s}$; 
\item
$H_t = G_{\tl{s}}$; 
\item
if $tv$ is $\H(\overline{F})$-conjugate to $\gamma$, and if $\gamma$ is a norm of $\tl{s}u$, 
then one can take $v = u^2$.
\end{itemize}
\end{prop}

Using this proposition, we can prove Theorem \ref{transfer}.
\begin{proof}[Proof of Theorem \ref{transfer}]
By Lemma \ref{descent} and Proposition \ref{bij}, 
it suffices to show that 
\[
SO_{u^2}(\1_{\K_{m,H} \cap H_t}) 
= SO_{u}(\1_{\K_{m,H} \cap H_t}) 
\]
for topologically unipotent elements $u \in H_t$
such that both $u$ and $u^2$ are strongly regular semisimple.
This follows from the fact that 
the map $u \mapsto u^2$ is a homeomorphism 
from the subset of topologically unipotent elements in $H_t$ to itself
(see \cite[Lemma 3.2.7]{GV}).
\end{proof}

%\section{Norm correspondence for topologically semisimple elements}
%\section{Norm correspondence for topologically semisimple elements}
\section{Norm correspondence for topologically semisimple elements}\label{sec.comparison}
The goal of this section is to prove Proposition \ref{bij}.

%\section{Topologically semisimple elements in hyperspecial maximal compact subgroups}
\subsection{Topologically semisimple elements in hyperspecial maximal compact subgroups}
As in Section \ref{sec.groups}, 
let $\G = \Res_{E/F}(\GL_{N,E})$ and $\H = \G_\theta$. 
We fix a hyperspecial maximal compact subgroup $\K$ of $G$ satisfying $\theta(\K) = \K$ 
such that $\K_H = \K \cap H$ is also a hyperspecial maximal compact subgroup of $H$.
In this subsection, 
we consider conjugacy classes of topologically semisimple elements in $\tl\K = \K \rtimes \theta$ or in $\K_H$. 
To do this, we take a smooth integral model $\GG$ of $\G$ over $\oo_F$ 
such that  $\GG(\oo_F) = \K$. 
Since $\theta(\K) = \K$, one can extend $\theta$ to 
an automorphism of $\GG$ over $\oo_F$. 
Moreover $\HH = \{h \in \GG \;|\; \theta(h) = h\}$ is a hyperspecial smooth integral model of $\H$ over $\oo_F$.
\par

\begin{lem}\label{lem_hyper}
Let $\tl{s} = s \rtimes \theta \in \tl\K$ be a topologically semisimple element. 
Then there exists $t \in \K_H$ such that $t$ is a norm of $\tl{s}$. 
\end{lem}
\begin{proof}
We denote the images of $s$ and $\theta(s)$ in $\HH(\overline{\oo_F/\pp_F})$ by 
$\overline{s}$ and $\overline{\theta(s)}$, respectively. 
Let $\overline{\theta}$ be the Frobenius map of $\HH(\overline{\oo_F/\pp_F})$, 
i.e., $\overline{\theta}((h_{i,j})_{i,j}) = J_N{}^t(h_{i,j}^q)_{i,j}^{-1}J_N^{-1}$. 
Then $\overline{\theta}(\overline{s}) = \overline{\theta(s)}$ 
and $\overline{\theta}(\overline{\theta(s)}) = \overline{s}$
(although $\overline{\theta}$ is not an involution).
By Lang's theorem, 
there exists $h \in \HH(\overline{\oo_F/\pp_F})$ such that 
$h^{-1} \overline{\theta}(h) = \overline{s}$.
If we set $\overline{\gamma} = h \overline{s} \overline{\theta(s)} h^{-1}$, 
then 
\begin{align*}
\overline{\theta}(\overline{\gamma}) 
= \overline{\theta}(h) \cdot \overline{\theta(s)} \overline{s} \cdot \overline{\theta}(h)^{-1}
= h \overline{s} \cdot \overline{\theta(s)} \overline{s} \cdot \overline{s}^{-1} h^{-1}
= \overline{\gamma}.
\end{align*}
Hence $\overline{\gamma} \in \HH(\oo_F/\pp_F)$. 
Take an arbitrary representative $\gamma \in \HH(\oo_F)$ of $\overline{\gamma}$.
Write $\gamma = tv$ for its TJD, 
and $\overline{t}, \overline{v}$ for the images of $t,v$ in $\HH(\oo_F/\pp_F)$. 
Since both $\overline{\gamma} = h \overline{s} \overline{\theta(s)} h^{-1}$ 
and $\overline{t}^{-1}$ are of finite order prime to $p$, 
and since $\overline{t}^{-1}$ commutes with $\overline{\gamma}$, 
we see that $\overline{v}$ is trivial, i.e., $\overline{\gamma} = \overline{t}$.
Especially, $t$ has the same eigenpolynomial as $N(s) = s\theta(s)$. 
This means that $t$ is a norm of $\tl{s}$. 
\end{proof}

Let $\T$ be the diagonal maximal torus of $\G$ and $\TT$ be its standard integral model over $\oo_F$. 
Then $\T_\theta = \T \cap \H$ is the diagonal maximal torus of $\H$ 
with an integral model $\TT_\theta = \TT \cap \HH$.
For simplicity, from now on, we assume that 
$\K = t_1 \GL_N(\oo_E) t_1^{-1}$ for some $t_1 \in T$. 
In particular, $\GG_0 = t_1^{-1}\GG t_1$ 
is a smooth integral model of $\G$ such that $\GG_0(\oo_F) = \GL_N(\oo_E)$.
Moreover, if we put $t_2 = J_N t_1^{-1} J_N^{-1}$ so that $t_2 = \overline{\theta(t_1)}$, 
since $\theta(\K) = \K$, 
we have 
\[
t_2 \GL_N(\oo_E) t_2^{-1} = \overline{t_1} \GL_N(\oo_E) \overline{t_1}^{-1}.
\]
\par

The following is a key fact for the proof of Proposition \ref{bij}. 
\begin{prop}\label{71}
\begin{enumerate}
\item
Let $\tl{s}, \tl{s}' \in \tl\K$ be topologically semisimple elements. 
If $\tl{s}$ and $\tl{s}'$ are $\G(\overline{F})$-conjugate, 
then $\tl{s}$ and $\tl{s}'$ are $\K$-conjugate. 

\item
Let $t,t' \in \K_H$ be topologically semisimple elements. 
If $t$ and $t'$ are $\H(\overline{F})$-conjugate, 
then $t$ and $t'$ are $\K_H$-conjugate. 
\end{enumerate}
\end{prop}

The assertion (2) is \cite[Proposition 7.1]{K-stable}. 
The proof of (1) might be similar. 
For the sake of completeness, we give a detail of the proof of (1) in the rest of this subsection.
\par

For a finite extension $L$ of $F$, 
we denote the ring of integers of $L$ by $\oo_L$.
\par

\begin{lem}\label{smooth}
Let $L$ be a finite Galois extension of $F$ containing $E$. 
If $\tl{s} \in \TT(\oo_L) \rtimes \theta$ is topologically semisimple, 
then the centralizer subgroup scheme $\GG_{\tl{s},\oo_L}$
of $\tl{s}$ in $\GG_{\oo_L}$ is a smooth group scheme over $\oo_L$
with connected reductive fibers.
\end{lem}
\begin{proof}
We put $t = \tl{s}^2$. 
This is an element of $\HH(\oo_L)$ which belongs to $\TT_\theta(\oo_L)$. 
By the same argument as in the proof of Proposition \ref{90} (1), 
we see that the centralizer subgroup schemes $\GG_{\tl{s},\oo_L}$ and $\HH_{t,\oo_L}$ 
are isomorphic over $\oo_L$. 
Since $t$ is topologically semisimple, 
by the argument in the second paragraph of the proof of \cite[Proposition 7.1]{K-stable}, 
we see that $\HH_{t,\oo_L}$ is a smooth group scheme over $\oo_L$ 
with connected reductive fibers.
(Note that the derived group of $\H$ is simply-connected.)
\end{proof}

When $L$ is a finite extension of $F$ containing $E$, 
the isomorphism $E \otimes_F L \cong L \times L$ induces an isomorphism 
$\G(L) \cong \GL_{N}(L) \times \GL_{N}(L)$
such that 
\begin{itemize}
\item
$\GG(\oo_L)$ is mapped to 
\[
t_1 \GL_N(\oo_L) t_1^{-1} \times \overline{t_1} \GL_N(\oo_L) \overline{t_1}^{-1}
= t_1 \GL_N(\oo_L) t_1^{-1} \times t_2 \GL_N(\oo_L) t_2^{-1};
\]

\item
$\TT(\oo_L)$ (\resp $\T(L)$) is mapped to 
$\T_{N}(\oo_L) \times \T_{N}(\oo_L)$ (\resp $\T_{N}(L) \times \T_{N}(L)$), 
where $\T_{N}$ is the diagonal maximal torus of $\GL_N$; 

\item
the automorphism $\theta$ is given by 
$(g_1,g_2) \mapsto (J_{N}{}^tg_2^{-1}J_{N}^{-1}, J_{N}{}^tg_1^{-1}J_{N}^{-1})$.
\end{itemize}
Note that, if we put $t_0 = (t_{1},t_{2})$, then $\GG(\oo_L) = t_0 \GG_0(\oo_L) t_0^{-1}$.
Moreover, since $t_0$ is $\theta$-invariant, 
we have $t_0 (\TT(\oo_L) \rtimes \theta) t_0^{-1} = \TT(\oo_L)$ 
and $t_0 (\GG_0(\oo_L) \rtimes \theta) t_0^{-1} = \GG(\oo_L) \rtimes \theta$.

\begin{lem}\label{ext-conjugate}
Suppose that $L$ is a finite Galois extension of $F$ containing $E$. 
Let $\tl{s} \in \TT(\oo_L) \rtimes \theta$ and $\tl{s}' \in \GG(\oo_L) \rtimes \theta$
be topologically semisimple elements. 
If $\tl{s}$ and $\tl{s}'$ are $\G(L)$-conjugate, then $\tl{s}$ and $\tl{s}'$ are $\GG(\oo_L)$-conjugate.
\end{lem}
\begin{proof}
We claim that it suffices to show the assertion when $\GG = \GG_0$. 
Indeed, if $\tl{s} \in \TT(\oo_L) \rtimes \theta$ and $\tl{s}' \in \GG(\oo_L) \rtimes \theta$ are $\G(L)$-conjugate, 
then $t_0^{-1}\tl{s}t_0 \in \TT(\oo_L) \rtimes \theta$ and $t_0^{-1}\tl{s}'t_0 \in \GG_0(\oo_L) \rtimes \theta$
are $\G(L)$-conjugate.
If we were to find $y \in \GG_0(\oo_L)$ 
such that $t_0^{-1}\tl{s}'t_0 = y (t_0^{-1}\tl{s}t_0) y^{-1}$, 
then $\tl{s}' = (t_0 y t_0^{-1}) \tl{s} (t_0 y t_0^{-1})^{-1}$ with $t_0 y t_0^{-1} \in \GG(\oo_L)$.
Thus we may assume that $\GG = \GG_0$ so that $t_0 = (t_1,t_2)$ is trivial. 
\par

Let us suppose that $x \in \G(L)$ satisfies $\tl{s}' = x\tl{s}x^{-1}$. 
Let $\B = \T\U$ be the upper triangular Borel subgroup of $\G$ with the unipotent radical $\U$.
By the Iwasawa decomposition, 
we may write $x = kut$, where $k \in \GG(\oo_L)$, $u \in \U(L)$ and $t \in \T(L)$.
By replacing $\tl{s}'$ with $k\tl{s}'k^{-1}$, 
we may assume that $k$ is trivial. 
\par

Since $ut\tl{s}t^{-1}u^{-1} = \tl{s}' \in \GG(\oo_L)$, 
by looking at the diagonal entries, 
we see that $t$ must belong to $\TT(\oo_L) \cdot \T_\theta(L)$.
If we write $t = t't''$ with $t' \in \TT(\oo_L)$ and $t'' \in \T_\theta(L)$, 
since $\tl{s} \in \TT(\oo_L) \rtimes \theta$ is commutative with $t''$, 
we have 
\[
\tl{s}' = ut\tl{s}t^{-1}u^{-1} 
= ut'\tl{s}t'^{-1}u^{-1} 
= t' \cdot (t'^{-1}ut') \cdot \tl{s} \cdot (t'^{-1}ut')^{-1} \cdot t'^{-1}.
\]
Hence, by replacing $\tl{s}'$ with $t'\tl{s}' t'^{-1}$ 
and $u$ with $t'ut'^{-1}$, respectively, 
we may assume that $t$ is trivial, i.e., $\tl{s}' = u \tl{s} u^{-1}$ for $u \in \U(L)$. 
\par

To get the assertion, it is enough to show that 
$u$ belongs to $\UU(\oo_L)\G_{\tl{s}}(L)$, 
where $\UU$ is the standard integral model of $\U$.
Since both $\tl{s}' = u\tl{s}u^{-1}$ and $\tl{s}$ belong to $\GG(\oo_L)$, 
we have $u\tl{s}u^{-1}\tl{s}^{-1} \in \GG(\oo_L)$.
If we write $\tl{s} = s \rtimes \theta$ with $s \in \TT(\oo_L)$, 
then we have
\[
u \tl{s} u^{-1} \tl{s}^{-1} 
= u \cdot s \theta(u^{-1}) s^{-1} \in \GG(\oo_L).
\]
\par

Since $L$ contains $E$, 
we have an isomorphism $\GG(\oo_L) \cong \GL_{N}(\oo_L) \times \GL_{N}(\oo_L)$ as above. 
Note that it furthermore satisfies that 
$\UU(\oo_L)$ is mapped to $\U_{N}(\oo_L) \times \U_{N}(\oo_L)$, 
where $\U_{N}$ is the upper triangular unipotent  subgroup of $\GL_{N}$.
Then, by letting $s = (s_1, s_2)$ and $u = (u_1, u_2)$, 
we have 
\[
u \cdot s \theta(u^{-1}) s^{-1} 
= (u_1 \cdot s_1 J_{N}{}^tu_2J_{N}^{-1}s_1^{-1}, u_2 \cdot s_2 J_{N}{}^tu_1J_{N}^{-1}s_2^{-1}).
\]
\par

Write $u_1 = (u_{1,i,j})_{i,j}$ and $u_2 = (u_{2,i,j})_{i,j}$. 
Then, for $1 \leq k \leq N-1$, 
the $(k,k+1)$-entry of $u_1 \cdot s_1 J_{N}{}^tu_2J_{N}^{-1}s_1^{-1}$
is given by 
\[
u_{1,k,k+1} - \alpha_{k,k+1}(s_1) u_{2,N-k,N+1-k}, 
\]
where $\alpha_{k,k+1}$ denotes the root of $\T_{N}$ corresponding to the $(k,k+1)$-entry.
Similarly, the $(N-k,N+1-k)$-entry of $u_2 \cdot s_2 J_{N}{}^tu_1J_{N}^{-1}s_2^{-1}$
is given by 
\[
u_{2,N-k,N+1-k} - \alpha_{N-k, N+1-k}(s_2) u_{1,k,k+1}.
\]
Thus the condition that $u\tl{s}u^{-1}\tl{s}^{-1} \in \GG(\oo_L)$
implies that 
\begin{enumerate}
\item
$u_{1,k,k+1} - \alpha_{k,k+1}(s_1) u_{2,N-k,N+1-k} \in \oo_L$; and 
\item
$u_{2,N-k,N+1-k} - \alpha_{N-k, N+1-k}(s_2) u_{1,k,k+1} \in \oo_L$.
\end{enumerate}
Since $\tl{s}$ is topologically semisimple, 
$\tl{s}^2 = (s_1 J_{N}{}^ts_2^{-1}J_{N}^{-1}, s_2 J_{N}{}^ts_1^{-1}J_{N}^{-1})$
is of finite order prime to $p$. 
Hence $\alpha_{k,k+1}(s_1 J_{N}{}^ts_2^{-1}J_{N}^{-1})$, 
which equals $\alpha_{k,k+1}(s_1) \alpha_{N-k,N+1-k}(s_2)$, 
is a root of unity of order prime to $p$.
In particular, 
one of 
\begin{itemize}
\item
$\alpha_{k,k+1}(s_1) \alpha_{N-k,N+1-k}(s_2) = 1$; or 
\item
$1-\alpha_{k,k+1}(s_1) \alpha_{N-k,N+1-k}(s_2) \in \oo_L^\times$
\end{itemize}
holds.
\par

In the latter case, the matrix 
\[
\begin{pmatrix}
1 & -\alpha_{k,k+1}(s_1) \\
-\alpha_{N-k,N+1-k}(s_2) & 1
\end{pmatrix}
\]
belongs to $\GL_2(\oo_L)$. 
Hence the above two conditions (1) and (2)
imply that both $u_{1,k,k+1}$ and $u_{2,N-k,N+1-k}$ belong to $\oo_L$. 
Now we consider the former case, i.e., we assume that $\alpha_{k,k+1}(s_1) \alpha_{N-k,N+1-k}(s_2) = 1$.
In this case, 
if we define $v_1 = (v_{1,i,j})_{i,j}, v_2 = (v_{2,i,j})_{i,j} \in \U_{N}(L)$ by
$v_{1,i,j} = v_{2,N+1-j,N+1-i} = \delta_{i,j}$ unless $(i,j) = (k,k+1)$ 
(where $\delta_{i,j}$ is the Kronecker delta)
and by 
\[
v_{1,k,k+1} = u_{1,k,k+1}, 
\quad
v_{2,N-k,N+1-k} = \alpha_{N-k, N+1-k}(s_2) u_{1,k,k+1},
\]
then $v = (v_1,v_2) \in \G_{\tl{s}}(L)$. 
Moreover, the $(k,k+1)$-entry of $u_1v_1^{-1}$ equals $0$, 
and the $(N-k,N+1-k)$-entry of $u_2v_2^{-1}$ is given by 
$u_{2,N-k,N+1-k} - \alpha_{N-k, N+1-k}(s_2) u_{1,k,k+1}$, 
which lies in $\oo_L$ by (2). 
\par

By this observation, 
we see that $u$ belongs to $\UU(\oo_L) \cdot [\U(L), \U(L)] \cdot \G_{\tl{s}}(L)$, 
where $[\U(L), \U(L)]$ denotes the commutator subgroup of $\U(L)$.
By applying the same argument to the entries of $u$ 
appearing in $[\U(L), \U(L)] / [\U(L), [\U(L), \U(L)]]$, 
we next see that $u$ belongs to 
$\UU(\oo_L) \cdot [\U(L), [\U(L), \U(L)]] \cdot \G_{\tl{s}}(L)$. 
Repeating this procedure, we eventually conclude that 
$u$ is in fact an element of $\UU(\oo_L)\G_{\tl{s}}(L)$.
This completes the proof.
\end{proof}

Now we are ready to prove Proposition \ref{71} (1). 
\begin{proof}[Proof of Proposition \ref{71} (1)]
Let $\tl{s}, \tl{s}' \in \GG(\oo_F) \rtimes \theta$ be topologically semisimple elements. 
Suppose that $\tl{s}$ and $\tl{s}'$ are $\G(\overline{F})$-conjugate. 
We take a finite Galois extension $L$ of $F$ such that
\begin{itemize}
\item
$\tl{s}$ and $\tl{s}'$ are $\G(L)$-conjugate; 
\item
$L$ contains $E$; and 
\item
we can find an element $\tl{s}'' \in \TT(\oo_L) \rtimes \theta$
which is $\G(L)$-conjugate to both $\tl{s}$ and $\tl{s}'$.
\end{itemize}
Then, by Lemma \ref{ext-conjugate}, 
$\tl{s}$, $\tl{s}'$ and $\tl{s}''$ are $\GG(\oo_L)$-conjugate to each other. 
\par

Let us consider the closed subscheme $\YY$ of $\GG$ defined over $\oo_F$
whose set of $R$-valued points is given by 
\[
\YY(R) = \{g \in \GG(R) \;|\; g \tl{s}g^{-1} = \tl{s}' \}
\]
for any $\oo_F$-algebra $R$. 
Since $\tl{s}$, $\tl{s}'$ and $\tl{s}''$ are $\GG(\oo_L)$-conjugate to each other, 
we see that $\YY_{\oo_L}$ and $\GG_{\tl{s}, \oo_L}$ are isomorphic to $\GG_{\tl{s}'',\oo_L}$
as schemes over $\oo_L$.
Since $\GG_{\tl{s}'',\oo_L}$ is smooth over $\oo_L$ with connected reductive fibers by Lemma \ref{smooth}, 
so are $\YY_{\oo_L}$ and $\GG_{\tl{s}, \oo_L}$. 
Thus the faithfully-flatness of $\oo_L/\oo_F$ implies that 
$\YY$ and $\GG_{\tl{s}, \oo_F}$ are smooth over $\oo_F$ with connected reductive fibers.
\par

We write $\tl{s} = s \rtimes \theta$ and $\tl{s}' = s' \rtimes \theta$ with $s,s' \in \GG(\oo_F)$. 
Let $\overline{s}$ and $\overline{s}'$ be the images of $s$ and $s'$ 
under the reduction map $\GG(\oo_F) \rightarrow \GG(\oo_F/\pp_F)$, respectively. 
Since $\tl{s}$ and $\tl{s}'$ are $\GG(\oo_L)$-conjugate, 
we see that $\overline{s} \rtimes \theta$ and $\overline{s}' \rtimes \theta$
are $\GG(\oo_L/\pp_L)$-conjugate. 
Thus the standard argument via Lang's theorem 
on the vanishing of the first Galois cohomology of $\GG_{\tl{s},\oo_F/\pp_F}$
implies that $\overline{s} \rtimes \theta$ and $\overline{s}' \rtimes \theta$
are $\GG(\oo_F/\pp_F)$-conjugate. 
In other words, $\YY(\oo_F/\pp_F)$ is non-empty. 
Hence, by the smoothness of $\YY$ over $\oo_F$, 
we see that $\YY(\oo_F)$ is also non-empty. 
Therefore 
$\tl{s}$ and $\tl{s}'$ are $\GG(\oo_F)$-conjugate.
This completes the proof of Proposition \ref{71} (1).
\end{proof}

%\subsection{Conjugacy classes of topologically semisimple elements}
\subsection{Conjugacy classes of topologically semisimple elements}
Suppose from now that $N = 2n+1$ is odd.
Fix a positive integer $m > 0$. 
Recall that $\K_m$ is the subgroup of $G$ consisting of matrices $k$ of the form
\[
\bordermatrix{
&n&1&n \cr
n&\oo_E&\oo_E&\pp_E^{-m}\cr
1&\pp_E^m&1+\pp_E^m&\oo_E\cr
n&\pp_E^m&\pp_E^m&\oo_E
}
\]
with $\det(k) \in \oo_E^\times$.
Note that $\theta(\K_m) = \K_m$.
Set $\K_{m,H} = \K_m \cap H$.
\par

\begin{prop}\label{eigen=1}
\begin{enumerate}
\item
Let $\tl{s} \in \tl{\K}_m$ be a topologically semisimple element. 
Then there exists $k \in \K_m$ such that
\[
k^{-1} \tl{s} k \in 
\bordermatrix{
&n&1&n \cr
n&\oo_E&0&\pp_E^{-m}\cr
1&0&1&0\cr
n&\pp_E^m&0&\oo_E
} \rtimes \theta.
\]
\item
Let $t \in \K_{m,H}$ be a topologically semisimple element. 
Then there exists $k \in \K_{m,H}$ such that
\[
k^{-1} t k \in 
\bordermatrix{
&n&1&n \cr
n&\oo_E&0&\pp_E^{-m}\cr
1&0&1&0\cr
n&\pp_E^m&0&\oo_E
}.
\]
\end{enumerate}
\end{prop}
\begin{proof}
To show (1), we prepare some notation. 
Let $\rho \colon G \rtimes \pair{\theta} \hookrightarrow \GL_{2N}(E)$ 
be a faithful representation given by 
\begin{align*}
\rho(g) = 
\begin{pmatrix}
g & 0 \\ 0 & \theta(g)
\end{pmatrix},
\quad
\rho(g \rtimes \theta) =
\begin{pmatrix}
0 & g \\
\theta(g) & 0
\end{pmatrix}.
\end{align*}
Set $\tl{V} = E^{2N}$ to be the representation space 
equipped with the canonical basis denoted by 
$e_{-n},\dots, e_{n}, f_{n}, \dots, f_{-n}$.
Let $\tl{L}$ be the $\oo_E$-lattice of $\tl{V}$ spanned by 
\[
\varpi^{-m} e_{-n},\dots, \varpi^{-m} e_{-1}, e_{0}, \dots, e_n,
\varpi^{-m} f_{n},\dots, \varpi^{-m} f_{1}, f_{0}, \dots, f_{-n}.
\]
Then $\rho(\K_m \rtimes \pair{\theta})$ preserves $\tl{L}$.
We denote by $V$ (\resp by $V'$) the subspace of $\tl{V}$
generated by $e_{-n},\dots, e_n$ (\resp by $f_{n},\dots, f_{-n}$). 
Set $L = \tl{L} \cap V$ and $L' = \tl{L} \cap V'$. 
Hence $\tl{V} = V \oplus V'$ and $\tl{L} = L \oplus L'$.
Define a hermitian pairing $\pair{\cdot, \cdot} \colon \tl{V} \times \tl{V} \rightarrow E$ so that 
$V$ and $V'$ are totally isotropic subspaces, and so that 
\[
\pair{
\sum_{i=-n}^n a_ie_i, 
\sum_{i=-n}^n b_i f_i 
}
= \sum_{i=-n}^n (-1)^{i-n} a_i \overline{b_i}
\]
for $a_i, b_i \in E$. 
Then $\rho(G)$ is characterized by 
the subgroup of $\GL(\tl{V})$ consisting of linear operators $f \colon \tl{V} \rightarrow \tl{V}$ 
such that $f(V) \subset V$, $f(V') \subset V'$ 
and such that $\pair{f(v),f(v')} = \pair{v,v'}$ for $v, v' \in \tl{V}$.
Moreover, one can check that 
$\pair{\rho(g \rtimes \theta)v, \rho(g \rtimes \theta)v'} = \pair{v,v'}$ for $v, v' \in \tl{V}$.
\par

Let $\tl{s} = s \rtimes \theta \in \tl{\K}_m$ be a topologically semisimple element. 
We consider $N(s) = s \theta(s) = (s \rtimes \theta)^2$.
It gives an $\oo_E$-linear automorphism of $L$.
Note that 
\[
\rho(N(s))e_0 \equiv e_0 \bmod \varpi^m L.
\]
Let $V_1$ (\resp $(L/\varpi^m L)_1$) be the eigenspace of 
$\rho(N(s))$ (\resp $\rho(N(s)) \bmod \varpi^m L$) with respect to the eigenvalue $1$.
Since $N(s)$ is of finite order prime to $p$, 
we see that 
$(L/\varpi^m L)_1$ is the image of $V_1 \cap L$ by the canonical projection $L \twoheadrightarrow L/\varpi^m L$.
In particular, 
there exists $v_0 \in V_1 \cap L$ such that $v_0 \equiv e_0 \bmod \varpi^m L$. 
\par

Set $v_0' = \rho(s \rtimes \theta)v_0$.
Then we have $v_0' \in L'$ and $v_0' \equiv f_0 \bmod \varpi^m L'$.
In particular, $(-1)^n\pair{v_0,v_0'} \in 1+\pp_E^m$ since $\pair{e_0,f_0} = (-1)^n$.
On the other hand, since $\pair{\cdot, \cdot}$ is $\rho(s \rtimes \theta)$-invariant, 
we have
\begin{align*}
\pair{v_0,v_0'} &= \pair{\rho(N(s))v_0, \rho(s \rtimes \theta)v_0}
\\&= \pair{\rho(s \rtimes \theta)v_0,v_0}
= \pair{v_0',v_0} = \overline{\pair{v_0,v_0'}}.
\end{align*}
Hence $(-1)^n\pair{v_0,v_0'} \in (1+\pp_E^m) \cap F = 1+\pp_F^m = N_{E/F}(1+\pp_E^m)$.
In particular, we can find $z \in 1+\pp_E^m$ such that 
if we set $e_0' = z v_0$ and $f_0' = z v_0'$, 
then 
\begin{itemize}
\item
$e_0' \in \oo_E e_{-n} \oplus \dots \oplus \oo_E e_{-1} 
\oplus (1+\pp_E^m) e_0 \oplus \pp_E^m e_1 \oplus \dots \oplus \pp_E^m e_n$; 
\item 
$f_0' \in \oo_E f_{n} \oplus \dots \oplus \oo_E f_{1} 
\oplus (1+\pp_E^m) f_0 \oplus \pp_E^m f_{-1} \oplus \dots \oplus \pp_E^m f_{-n}$;
\item
$\pair{e_0', f_0'} = (-1)^n = \pair{e_0,f_0}$; 
\item
$f_0' = \rho(s \rtimes \theta)e_0'$ and $e_0' = \rho(s \rtimes \theta)f_0'$.
\end{itemize}
Fix $\epsilon \in \oo_E^\times$ such that $\overline{\epsilon} = -\epsilon$. 
For $a \in \oo_F^\times$, set 
\[
z_a = \frac{1+ a \epsilon \varpi^m}{1-a \epsilon \varpi^m}. 
\]
Then $N_{E/F}(z_a) = 1$ and $z_a \equiv 1+2a \epsilon \varpi^m \bmod \pp_E^{2m}$.
If necessary, 
by replacing $e_0'$ and $f_0'$ by $z_a e_0'$ and $z_a f_0'$ simultaneously for a suitable $a \in \oo_F^\times$, 
we may furthermore assume that 
\begin{itemize}
\item
$\pair{e_0,f_0'-f_0}, \pair{f_0,e_0'-e_0} \in \varpi^m \oo_E^\times$.
\end{itemize}
\par

Based on an argument of Witt's theorem, 
for $-n \leq j \leq n$ with $j \not= 0$, 
we set
\begin{align*}
e_j' &= e_j + \frac{\pair{e_j, f_0'-f_0}}{\pair{e_0, f_0'-f_0}}(e_0'-e_0), \\
f_j' &= f_j + \frac{\pair{f_j, e_0'-e_0}}{\pair{f_0, e_0'-e_0}}(f_0'-f_0). 
\end{align*}
We claim that $\pair{e_i', f_j'} = \pair{e_i,f_j}$ for any $-n \leq i,j \leq n$. 
Indeed, if both $i$ and $j$ are not equal to $0$, then 
\begin{align*}
&\pair{e_i', f_j'} - \pair{e_i,f_j} 
\\&= 
\frac{\pair{e_i,f_0'-f_0}}{\pair{e_0,f_0'-f_0}}
\overline{\left(\frac{\pair{f_j, e_0'-e_0}}{\pair{f_0, e_0'-e_0}}\right)}
\left(\pair{e_0,f_0'-f_0} + \overline{\pair{f_0,e_0'-e_0}} + \pair{e_0'-e_0, f_0'-f_0} \right)
\\&=
\frac{\pair{e_i,f_0'-f_0}}{\pair{e_0,f_0'-f_0}}
\overline{\left(\frac{\pair{f_j, e_0'-e_0}}{\pair{f_0, e_0'-e_0}}\right)}
\left( \pair{e_0',f_0'} - \pair{e_0,f_0} \right) = 0.
\end{align*}
Similarly, one can easily check that $\pair{e_j', f_0'} = \pair{e_0', f_j'} = 0$ for $j \not= 0$.
Therefore, 
the change-of-basis matrix from 
$(e_{-n},\dots, e_{n}, f_{n}, \dots, f_{-n})$
to $(e_{-n}',\dots, e_{n}', f_{n}', \dots, f_{-n}')$
is given by $\rho(k)$ for some $k \in G$. 
Moreover, by construction, we see that $k \in \K_{m}$.
Since the orthogonal complement of $\{e_0', f_0'\}$ in $\tl{V}$, 
which is spanned by $\{e_j', f_j'\}_{j \not= 0}$, 
is preserved by $\rho(s \rtimes \theta)$, 
we have
\[
k^{-1} (s \rtimes \theta) k \in 
\bordermatrix{
&n&1&n \cr
n&\oo_E&0&\pp_E^{-m}\cr
1&0&1&0\cr
n&\pp_E^m&0&\oo_E
} \rtimes \theta,
\]
as desired. 
\par

The proof of (2) is similar and easier. 
Let us give a brief sketch. 
\begin{itemize}
\item
We regard $V = E^N$ as a hermitian space so that $H = \u(V)$. 
Write $e_{-n}, \dots, e_n$ for the canonical basis. 

\item
By the same argument as in the proof of (1), 
one can take 
\[
e_0' \in \oo_Ee_{-n} \oplus \dots \oo_E e_{-1} \oplus (1+\pp_E^m)e_0 
\oplus \pp_E^m e_1 \oplus \dots \oplus \pp_E^m e_n
\]
such that $te_0' = e_0'$. 
Moreover, one can normalize $e_0'$ so that $\pair{e_0',e_0'} = \pair{e_0,e_0}$. 

\item
In addition, by replacing $e_0'$ with 
\[
\frac{1+\epsilon\varpi^m}{1-\epsilon\varpi^m}e'_0
\]
if necessary, we may furthermore assume that $\pair{e_0, e_0'-e_0} \in \varpi^m \oo_E^\times$.

\item
For $-n \leq j \leq n$ with $j \not= 0$, 
define $e_j'$ by 
\[
e_j' = e_j + \frac{\pair{e_j, e_0'-e_0}}{\pair{e_0, e_0'-e_0}}(e_0'-e_0). 
\]
It is easy to check that $\pair{e_i',e_j'} = \pair{e_i,e_j}$ for $-n \leq i,j \leq n$.

\item
The change-of-basis matrix from 
$(e_{-n},\dots, e_{n})$ to $(e_{-n}',\dots, e_{n}')$
gives an element $k \in \K_{m,H}$. 
This element satisfies the desired condition.
\end{itemize}
This completes the proof.
\end{proof}

\begin{cor}\label{leq1}
For $\gamma \in H^\srs$, 
we have $|I_\gamma(\K_m)| \leq 1$ and $|J_\gamma(\K_{m,H})| \leq 1$.
\end{cor}
\begin{proof}
The assertion $|J_\gamma(\K_{m,H})| \leq 1$
immediately follows from Propositions \ref{eigen=1} (2) and \ref{71} (2).
Similarly, the assertion $|I_\gamma(\K_m)| \leq 1$ 
is proven by using Propositions \ref{eigen=1} (1) and \ref{71} (1), 
but we have to check it carefully. 
\par

Let $\tl{s}_1 = s_1 \rtimes \theta, \tl{s}_2 = s_2 \rtimes \theta \in I_\gamma(\K_m)$. 
We will show that $\tl{s}_1$ is $\K_m$-conjugate to $\tl{s}_2$. 
By Proposition \ref{eigen=1} (1), 
we may assume that 
\[
\tl{s}_1,\tl{s}_2 \in 
\bordermatrix{
&n&1&n \cr
n&\oo_E&0&\pp_E^{-m}\cr
1&0&1&0\cr
n&\pp_E^m&0&\oo_E
} \rtimes \theta
\]
after taking $\K_m$-conjugations if necessary. 
\par

Recall from Section \ref{sec.groups} that
we have an inclusion $\iota \colon G' = \GL_{2n}(E) \hookrightarrow G = \GL_{2n+1}(E)$. 
If we set $\K'_m = \iota^{-1}(\K_m) \subset G'$, 
then we can write $s_1 = \iota(s_1'), s_2 = \iota(s_2')$ for some $s_1', s_2' \in \K'_m$.
Although $G$ and $G'$ have involutions both denoted by $\theta$, 
they are not compatible with respect to $\iota$. 
Indeed, $G'$ has another involution $\theta'$ 
such that $\iota(\theta'(x)) = \theta(\iota(x))$ for $x \in G'$. 
Two involutions $\theta$ and $\theta'$ on $G'$ are related by 
$\theta'(x) = \kappa \cdot \theta(\kappa^{-1}x\kappa) \cdot \kappa^{-1}$ for $x \in G'$, 
where $\kappa \in \K'_m$ is as in Section \ref{sec.groups}.
In particular,
\[
(s_i' \rtimes \theta')^2 = s_i' \theta'(s_i') 
= s_i' \kappa \theta(\kappa^{-1} s_i' \kappa) \kappa^{-1}
= \kappa(\kappa^{-1}s_i'\kappa \rtimes \theta)^2 \kappa^{-1}.
\]
Since $\iota((s_i' \rtimes \theta')^2) = (s_i \rtimes \theta)^2$, 
we see that $\kappa^{-1}s_i'\kappa \rtimes \theta$ 
is a topologically semisimple element in $G' \rtimes \theta$. 
Moreover, by considering the norm correspondence, 
there exists $g \in \G'(\overline{F})$ such that 
$s_2' \rtimes \theta' = g(s_1' \rtimes \theta')g^{-1}$. 
It is equivalent to saying that 
\begin{align*}
&s_2' = gs_1' \theta'(g^{-1}) = g s_1' \kappa \theta(\kappa^{-1}g^{-1}\kappa) \kappa^{-1}
\\&\iff 
\kappa^{-1}s_2'\kappa 
= (\kappa^{-1}g\kappa) \cdot \kappa^{-1}s_1'\kappa \cdot \theta(\kappa^{-1}g^{-1}\kappa). 
\end{align*}
Hence $\kappa^{-1}s_2'\kappa \rtimes \theta$ is $\G'(\overline{F})$-conjugate to 
$\kappa^{-1}s_1'\kappa \rtimes \theta$.
By Proposition \ref{71} (1), 
there exists $k' \in \K_m'$ such that 
$\kappa^{-1}s_2'\kappa \rtimes \theta = k'(\kappa^{-1}s_1'\kappa \rtimes \theta)k'^{-1}$.
By the same calculation as above, 
it is equivalent to saying that 
$s_2' \rtimes \theta' = 
(\kappa k' \kappa^{-1}) \cdot (s_1' \rtimes \theta') \cdot (\kappa k' \kappa^{-1})^{-1}$.
If we set $k = \iota(\kappa k' \kappa^{-1}) \in \K_m$, 
then we have $s_2 \rtimes \theta = k(s_1 \rtimes \theta)k^{-1}$, 
as desired. 
\end{proof}

\begin{lem}\label{A=>B}
If $J_\gamma(\K_{m,H}) \not= \emptyset$, 
then $I_\gamma(\K_{m}) \not= \emptyset$. 
Moreover, for $t \in J_\gamma(\K_{m,H})$, 
we can assume that $\tl{s} \in I_\gamma(\K_m)$ satisfies 
the three conditions in Proposition \ref{bij}.
\end{lem}
\begin{proof}
Let $t \in J_\gamma(\K_{m,H})$.
In Lemma \ref{lemA} below, 
we will show that
there exists $\tl{s} = s \rtimes \theta \in \tl\K_m$ with $s \in \oo_E[t]$ 
such that $t = N(s)$. 
In particular, $t$ is a norm of $\tl{s} = s \rtimes \theta$.
Since every $g \in \G_t$ commutes with $s$, 
by the same argument as in the proof of Proposition \ref{90} (2), 
we have $\G_{\tl{s}} = \H_t$ as algebraic subgroups of $\G$.
Let $v \in \K_{m,H} \cap H_t$ be a topologically unipotent element 
such that $tv$ is $\H(\overline{F})$-conjugate to $\gamma$. 
By the same argument as in \cite[Lemma 3.2.7]{GV}, 
there exists a unique topologically unipotent element $u \in H_t = G_{\tl{s}}$ such that $v = u^2$. 
Moreover, since $\K_{m}$ is closed in $G$ and $p$ is odd, 
we have $u \in \K_m \cap G_{\tl{s}}$.
Since $(\tl{s}u)^2 = s\theta(u)\theta(s)u = N(s)u^2 = tv$, 
we see that $tv$ is a norm of $\tl{s}u$. 
Therefore, we conclude that $\tl{s} \in I_\gamma(\K_{m})$.
\end{proof}

\begin{rem}\label{A=>B_easy}
For the proof of Lemma \ref{lemA},
we will generalize Hilbert's Theorem 90 by using the \emph{faithfully flat descent} (Proposition \ref{general90}).
However, if $q = |\oo_F/\pp_F| > N$, 
one can easily show the existence of 
$\tl{s} = s \rtimes \theta \in \tl\K_m$ with $s \in \oo_E[t]$ 
such that $t = N(s)$ as follows. 
\par

As in the proof of Proposition \ref{90} (2), 
we consider $s = (\alpha I_N + \overline{\alpha}t)/(\alpha+\overline{\alpha})$ 
for $\alpha \in \oo_E^\times$ with $\alpha+\overline{\alpha} \not= 0$.
It suffices to find $\alpha$ such that $\det(s) \in \oo_E^\times$.
If we denote the eigenpolynomial of $t$ by $\Phi_t$, 
we have
$\det(s) = (-\overline{\alpha}/(\alpha+\overline{\alpha}))^N \Phi_t(-\alpha/\overline{\alpha})$.
Hence it is enough to find $\alpha \in \oo_E^\times$ such that 
$\alpha+\overline{\alpha} \not\equiv 0 \bmod \pp_E$ and 
$\Phi_t(-\alpha/\overline{\alpha}) \not\equiv 0 \bmod \pp_E$. 
Since 
\[
|\{-\alpha/\overline{\alpha} \;|\; \alpha \in (\oo_E/\pp_E)^\times,\; \alpha+\overline{\alpha} \not= 0\}|
= \frac{(q^2-1)-(q-1)}{q-1} = q > N,
\]
we can find $\alpha \in \oo_E^\times$ satisfying the desired conditions.
\end{rem}

\begin{lem}\label{B=>A}
If $I_\gamma(\K_{m}) \not= \emptyset$, 
then $J_\gamma(\K_{m,H}) \not= \emptyset$. 
\end{lem}
\begin{proof}
Take $\tl{s} = s \rtimes \theta \in I_\gamma(\K_m)$. 
By Proposition \ref{eigen=1} (1), we may assume that 
$s = \iota(s')$ for some $s' \in \K_m' = \iota^{-1}(\K_m) \subset G'$.
As we have seen in the proof of Corollary \ref{leq1}, 
$\kappa^{-1}s'\kappa \rtimes \theta$ is topologically semisimple, 
where $\kappa \in \K_m'$ is as in Section \ref{sec.groups}. 
By Lemma \ref{lem_hyper}, 
there is $t' \in \K_m' \cap G'_{\theta}$ such that 
$t'$ is a norm of $\kappa^{-1}s'\kappa \rtimes \theta$. 
This means that there exists $g' \in \G'(\overline{F})$
such that $g't'g'^{-1} = \kappa^{-1}s'\kappa \theta(\kappa^{-1}s'\kappa)$. 
Since $\theta'(x) = \kappa \cdot \theta(\kappa^{-1}x\kappa) \cdot \kappa^{-1}$, 
we have 
$(\kappa g' \kappa^{-1})\cdot(\kappa t' \kappa^{-1})\cdot(\kappa g' \kappa^{-1})^{-1}
= s' \theta'(s')$.
If we set $t = \iota(\kappa t' \kappa^{-1})$ and 
$g = \iota(\kappa g' \kappa^{-1})$, 
then we have $t \in \K_{m,H}$ and $gtg^{-1} = s\theta(s)$.
Hence $t$ is a norm of $\tl{s}$, 
which is equivalent to saying that 
$t$ is $\H(\overline{F})$-conjugate to $\gamma_s$.
\par

By applying Lemma \ref{lemA} (or Remark \ref{A=>B_easy} if $q > N$) to $t \in \K_{m,H}$, 
one can obtain $\tl{s}_0 = s_0 \rtimes \theta \in \tl\K_{m}$ such that $t = s_0\theta(s_0)$ and $G_{\tl{s}_0} = H_t$. 
Since $\tl{s}, \tl{s}_0 \in \tl\K_m$, by the argument in Corollary \ref{leq1}, 
we have $k \in \K_m$ such that $\tl{s}_0 = k\tl{s}k^{-1}$.
Then $H_t = G_{\tl{s}_0} = kG_{\tl{s}}k^{-1}$.
\par

Suppose now that $\gamma$ is a norm of $\tl{s}u$ with $u \in \K_m \cap G_{\tl{s}}$. 
Set $u_0 = kuk^{-1} \in \K_m \cap G_{\tl{s}_0}$ and $v = u_0^2 \in \K_{m,H} \cap H_t$.
Then we have 
\begin{align*}
tv = s_0\theta(s_0)u_0^2 = (\tl{s}_0u_0)^2 = (k\tl{s}uk^{-1})^2 = k (\tl{s}u)^2 k^{-1}. 
\end{align*}
It means that $tv$ is a norm of $\tl{s}u$, 
which is equivalent to saying that 
$tv$ is $\H(\overline{F})$-conjugate to $\gamma$.
Therefore, we conclude that $t \in J_\gamma(\K_{m,H})$.
\end{proof}

Now Proposition \ref{bij} 
follows from Corollary \ref{leq1} together with Lemmas \ref{A=>B} and \ref{B=>A}.

%\section{Local Langlands correspondence and local newforms}
%\section{Local Langlands correspondence and local newforms}
\section{Local Langlands correspondence and local newforms}\label{sec.newforms}
In this section, we prove Theorem \ref{main1} as an application of Theorem \ref{transfer}. 
To do this, we recall the local Langlands correspondence.

%\subsection{Local Langlands correspondence for $\u_{2n+1}$}
\subsection{Local Langlands correspondence for $\u_{2n+1}$}
Let $W_E \subset W_F$ be the Weil groups of $E$ and $F$, respectively. 
Fix $s \in W_F \setminus W_E$. 
Set $\WD_E = W_E \times \SL_2(\C)$ to be the Weil--Deligne group of $E$. 
We call a representation $\phi \colon \WD_E \rightarrow \GL_N(\C)$
\emph{conjugate orthogonal}
if there exists a non-degenerate bilinear form $B \colon \C^N \times \C^N \rightarrow \C$ 
such that 
\[
\left\{
\begin{aligned}
&B(\phi(w)v, \phi(sws^{-1})v') = B(v,v'), \\
&B(v',v) = B(v,\phi(s^2)v')
\end{aligned}
\right. 
\]
for $v,v' \in \C^N$ and $w \in \WD_E$. 
\par

Set $H = \u_{2n+1}$ to be the quasi-split unitary group with $2n+1$ variables as in the previous sections. 
Let $\Phi_\temp(H)$ be the set of equivalence classes of 
conjugate orthogonal representations $\phi$ of $\WD_E$ of dimension $2n+1$
such that $\phi(W_E)$ is bounded, 
and let $\Irr_\temp(H)$ be the set of equivalence classes of 
irreducible tempered representations of $H$.
\par

For $\pi \in \Irr_\temp(H)$, 
we have the Harish-Chandra character $\Theta_\pi \colon C_c^\infty(H) \rightarrow \C$, 
which is defined by 
\[
\Theta_\pi(f^H) 
= \tr\left(v \mapsto
\int_{H} f^H(h) \pi(h)v dh
\right), \quad f^H \in C_c^\infty(H).
\]
On the other hand, for $\phi \in \Phi_\temp(H)$, 
let $\pi_\phi$ be the irreducible tempered representation of $G = \GL_{2n+1}(E)$ 
associated with $\phi$
by the local Langlands correspondence for $\GL_{2n+1}(E)$. 
Since $\pi_\phi$ is conjugate self-dual, 
we have a nonzero intertwining operator 
$I \colon \pi_\phi \xrightarrow{\sim} \pi_\phi^\theta$ with $I^2 = \id$. 
Note that $I$ is unique up to $\{\pm1\}$. 
We normalize $I$ by using the fixed $\theta$-stable $F$-splitting of $G$. 
Then we can consider the twisted trace
\[
\Theta_{\pi_\phi, \theta}(f)
= \tr\left(v \mapsto
\int_{G} f(g \rtimes \theta) (\pi(g) \circ I)v dg
\right), \quad f \in C_c^\infty(\tl{G}).
\]
\par

The local Langlands correspondence for $H$ is as follows.
\begin{thm}[{\cite[Theorem 3.2.1]{Mok}}]\label{LLC}
For $\phi \in \Phi_\temp(H)$, 
there exists a finite subset $\Pi_\phi$ of $\Irr_\temp(H)$ such that 
\[
\Irr_\temp(H) = \bigsqcup_{\phi \in \Phi_\temp(H)} \Pi_\phi. 
\]
Moreover, $\Pi_\phi$ is characterized by the endoscopic character relation 
\[
\Theta_{\pi_\phi, \theta}(f) = \sum_{\pi \in \Pi_\phi} \Theta_{\pi}(f^H)
\]
for any $f \in C_c^\infty(\tl{G})$ and $f^H \in C_c^\infty(H)$ such that $f^H$ is a transfer of $f$.
\end{thm}
We call $\Pi_\phi$ the \emph{$L$-packet} associated with $\phi$. 
When $\pi \in \Pi_\phi$, we say that $\phi$ is the \emph{$L$-parameter} for $\pi$.

%\subsection{Multiplicity one in $L$-packets}
\subsection{Multiplicity one in $L$-packets}\label{sec.mult1}
Fix a non-trivial additive character $\psi_E \colon E \rightarrow \C^\times$
such that $\psi_E|_{\oo_E} = \1$ but $\psi_E|_{\pp_E^{-1}} \not= \1$. 
For $\phi \in \Phi_\temp(H)$, 
let $\ep(s,\phi,\psi_E)$ be the $\ep$-factor associated with $\phi$ and $\psi_E$. 
We define the \emph{conductor} $c(\phi)$ of $\phi$ by 
the non-negative integer satisfying 
\[
\ep(s, \phi, \psi_{E}) = q^{c(\phi)(1-2s)}\ep(1/2, \phi, \psi_E). 
\]
Here, we note that $|\oo_E/\pp_E| = q^2$, 
and that $c(\phi)$ is independent of the choice of $\psi_E$.
\par

\begin{rem}
By \cite[Proposition 5.2 (2)]{GGP}, 
if $\psi_E$ is trivial on $F$, then $\ep(1/2, \phi, \psi_E) = 1$. 
\end{rem}

As an application of Theorem \ref{transfer}, we prove the following theorem.
\begin{thm}\label{packet=1}
For $\phi \in \Phi_\temp(H)$, we have 
\[
\sum_{\pi \in \Pi_\phi} \dim(\pi^{\K_{m,H}}) = 
\left\{
\begin{aligned}
&0 \iif m < c(\phi), \\
&1 \iif m = c(\phi).
\end{aligned}
\right. 
\]
Moreover, 
if $m > c(\phi)$, then the left hand side is nonzero.
\end{thm}
\begin{proof}
Note that for $\pi \in \Irr_\temp(H)$, 
we have
\[
\dim(\pi^{\K_{m,H}}) = \Theta_\pi(\vol(\K_{m,H}; dh)^{-1}\1_{\K_{m,H}}).
\]
Hence by Theorems \ref{transfer} and \ref{LLC}, 
we have 
\[
\sum_{\pi \in \Pi_\phi} \dim(\pi^{\K_{m,H}}) = 
\Theta_{\pi_\phi, \theta}(\vol(\K_m; dg)^{-1}\1_{\tl\K_m}).
\]
\par

Since $\pi(g) \circ I = I \circ \pi(\theta(g))$ and $\theta(\K_m) = \K_m$, 
we see that $I$ preserves $\pi_\phi^{\K_m}$. 
Moreover, the image of the operator 
\[
v \mapsto \vol(\K_m; dg)^{-1} \int_{G} \1_{\tl\K_m}(g \rtimes \theta) (\pi(g) \circ I)v dg
\]
is equal to $\pi_\phi^{\K_m}$, and 
the restriction of this operator to $\pi_{\phi}^{\K_m}$ coincides with $I$.
Hence 
\[
\Theta_{\pi_\phi, \theta}(\vol(\K_m; dg)^{-1}\1_{\tl\K_m}) = \tr\left(I; \pi_\phi^{\K_m}\right).
\]
\par

Since
\[
\K_m = 
\begin{pmatrix}
I_n && \\
&1& \\
&&\varpi^m I_n
\end{pmatrix}
\begin{pmatrix}
\oo_E & \oo_E & \oo_E \\
\pp_E^m & 1+\pp_E^m & \pp_E^m \\
\oo_E & \oo_E & \oo_E
\end{pmatrix}
\begin{pmatrix}
I_n && \\
&1& \\
&&\varpi^m I_n
\end{pmatrix}^{-1},
\]
by \cite[(5.1) Th{\'e}or{\`e}me]{JPSS}, 
we have
\[
\dim(\pi_\phi^{\K_m}) = 
\left\{
\begin{aligned}
&0 \iif m < c(\phi), \\
&1 \iif m = c(\phi).
\end{aligned}
\right. 
\]
In particular, if $m < c(\phi)$, 
then 
\[
\sum_{\pi \in \Pi_\phi} \dim(\pi^{\K_{m,H}}) = \tr\left(I; \pi_\phi^{\K_m}\right) = 0.
\]
On the other hand, if $m = c(\phi)$, then 
\[
\sum_{\pi \in \Pi_\phi} \dim(\pi^{\K_{m,H}}) = \tr\left(I; \pi_\phi^{\K_m}\right) \in \{\pm1\}.
\]
Since the left hand side is a non-negative integer, it should be equal to $1$.
\par

Similarly, since $\dim(\pi_{\phi}^{\K_{c(\phi)+1}}) = 2n+1$ by \cite[(2.2) Theorem 1]{Reeder}, 
and since $I^2 = \id$,
we see that $\tr(I; \pi_\phi^{\K_{c(\phi)+1}}) \not= 0$ 
so that $\sum_{\pi \in \Pi_\phi} \dim(\pi^{\K_{c(\phi)+1,H}}) > 0$.
Since
\[
\begin{pmatrix}
\varpi I_n && \\
&1& \\
&&\varpi^{-1} I_n
\end{pmatrix}
\K_{m+2, H}
\begin{pmatrix}
\varpi I_n && \\
&1& \\
&&\varpi^{-1} I_n
\end{pmatrix}^{-1}
\subset \K_{m,H}, 
\]
we have $\dim(\pi^{\K_{m+2,H}}) \geq \dim(\pi^{\K_{m,H}})$ for any $m \geq 0$. 
Therefore, we conclude that 
\[
\sum_{\pi \in \Pi_\phi} \dim(\pi^{\K_{c(\phi)+m,H}}) > 0
\] 
for any $m > 0$.
\end{proof}

%\subsection{Generic representations}
\subsection{Generic representations}
Let $U \subset H$ be the subgroup consisting of upper triangular unipotent matrices. 
By abuse of notation, we denote the character
\[
U \ni u = (u_{i,j})_{i,j} \mapsto \psi_{E}(u_{1,2}+\dots+u_{n,n+1}) \in \C^\times
\]
by $\psi_E$. 
We say that an irreducible representation $\pi$ of $H$ is \emph{generic}
if $\Hom_{U}(\pi, \psi_E) \not= 0$. 
\par

Recall that for $\phi \in \Phi_\temp(H)$, 
the $L$-packet $\Pi_\phi$ contains a unique generic representation 
(see \cite[Corollary 9.2.4]{Mok} and \cite[Theorem 3.1]{At}). 
Now Theorem \ref{main1} follows from Theorem \ref{packet=1} and the following result. 

\begin{thm}\label{generic}
Let $\pi \in \Irr_\temp(H)$. 
If $\pi^{\K_{m,H}} \not= 0$ for some $m \geq 0$, 
then $\pi$ is generic.
\end{thm}
\begin{proof}
As in Section \ref{sec.groups}, 
we consider another unramified unitary group
\[
H' = G'_{\theta'} = \{g' \in \GL_{2n}(E) \;|\; g' J_{2n}' {}^t\overline{g'} = J_{2n}'\}.
\]
Via the inclusion $\iota \colon H' \hookrightarrow H$ defined in that subsection, 
we regard $H'$ as a subgroup of $H$. 
Then $\K_{m,H} \cap H'$ is a hyperspecial maximal compact subgroup of $H'$. 
\par

Let $\pi \in \Irr_\temp(H)$ be such that $\pi^{\K_{m,H}} \not= 0$. 
We claim that there exists an irreducible tempered unramified representation $\pi'$ of $H'$ 
such that $\Hom_{H'}(\pi \otimes \pi', \C) \not= 0$. 
If this claim were to be shown, 
by the local Gan--Gross--Prasad conjecture (\cite[Conjecture 17.3]{GGP})
proven by Beuzart-Plessis \cite{BP1,BP2,BP3}, 
$\pi$ would be uniquely determined by the $L$-parameters for $\pi$ and $\pi'$. 
For the $L$-parameter for $\pi'$, see e.g., \cite[Sections 10]{GGP}. 
Since $\pi'$ is unramified, 
and since there exists only one unramified character $\chi$ of $E^\times$ 
such that $\chi|_{F^\times}$ is equal to the quadratic character associated to $E/F$, 
the $L$-parameter $\phi'$ for $\pi'$ should be of the form $\phi_1 \oplus {}^c\phi_1^\vee$
for some representation $\phi_1$ of $\WD_E$, 
where ${}^c\phi_1^\vee$ is the conjugate dual of $\phi_1$. 
Using \cite[Proposition 5.1 (2)]{GGP} and \cite[Corollary 9.2.4]{Mok}, 
we could conclude that $\pi$ is generic. 
\par

The claim is essentially the same as a lemma of Gan--Savin (\cite[Lemma 12.5]{GS}).
Fix $v \in \pi^{\K_{m,H}}$ such that $v \not= 0$. 
Let $\pair{\cdot,\cdot}$ be an $H$-invariant inner product on $\pi$. 
Then $f_\pi(h) = \pair{\pi(h)v,v}$ is a matrix coefficient of $\pi$. 
Moreover, it is bi-$\K_{m,H}$-invariant and $f_\pi(I_{2n+1}) \not= 0$.
Since $f_\pi|_{H'} \not= 0$, 
by the same argument as in the proof of \cite[Lemma 12.5]{GS}, 
we can find an irreducible tempered representation $\pi'$ of $H'$
and a matrix coefficient $f_{\pi'}$ of $\pi'$ such that 
\[
\int_{H'} f_{\pi}(h') f_{\pi'}(h') dh' \not= 0. 
\]
Here, the absolutely convergence of double integrals which we need 
was proven in \cite[Proposition 2.1]{H}. 
However, since $f_\pi$ is bi-$\K_{m,H}$-invariant, 
we can take $f_{\pi'}$ so that $f_{\pi'}$ is bi-$(\K_{m,H} \cap H')$-invariant. 
Then $\pi'$ must be unramified and $\Hom_{H'}(\pi \otimes \pi', \C) \not= 0$, as desired. 
\end{proof}

\appendix
%\section{Faithfully flat descent and Hilbert's Theorem 90}\label{appA}
%\section{Faithfully flat descent and Hilbert's Theorem 90}\label{appA}
\section{Faithfully flat descent and Hilbert's Theorem 90}\label{appA}
In this appendix, we generalize Hilbert's Theorem 90 (Proposition \ref{general90}) 
by using the faithfully flat descent. 
Using this proposition, we will prove Lemma \ref{lemA}, 
which was used in the proof of Lemma \ref{A=>B}.

%\subsection{A generalization of Hilbert's Theorem 90}
\subsection{A generalization of Hilbert's Theorem 90}
Now we prove a generalization of Hilbert's Theorem 90. 

\begin{prop}\label{general90}
Let $f \colon A \rightarrow B$ 
be a ring homomorphism between commutative rings $A$ and $B$ with the unit element $1$, 
and 
let $G$ be a cyclic group of order $n$ with a generator $\sigma \in G$. 
Suppose that $G$ acts on $B$ from the left, 
and that 
\begin{enumerate}
\item
$f$ is faithfully flat, i.e., 
$B$ is a faithfully flat $A$-module; 
\item
$f(A)$ is contained in the $G$-invariant part $B^G$, 
which means that every $g \in G$ acts on $B$ as an $A$-homomorphism $B \rightarrow B$;
\item
the homomorphism $\phi \colon B \otimes_A B \rightarrow \prod_{g \in G} B$ 
defined by $b \otimes c \mapsto (b g(c))_{g \in G}$
is an isomorphism; 
\item
$\Pic(A) = 0$. 
\end{enumerate}
Then for any $x \in B$ with $\prod_{g \in G}g(x) = 1$, 
there exists $y \in B^\times$ such that $x = y/\sigma(y)$.
\end{prop}
\begin{proof}
Remark that the homomorphism $\phi$ in (3) is well-defined by (2). 
Consider 
\[
M = \{y \in B \;|\; \sigma(y) \cdot x = y\}. 
\]
It is an $A$-submodule of $B$. 
We have to show that $M \cap B^\times$ is not empty. 
\par

Since $B$ is a flat $A$-module, 
we have 
\[
M \otimes_A B = \{y_B \in B \otimes_A B \;|\; (\sigma \otimes \id_B)(y_B) \cdot (x \otimes 1) = y_B\}.
\]
We translate it using the isomorphism $\phi \colon B \otimes_A B \rightarrow \prod_{g \in G} B$.  
If we define an action of $\sigma$ on $\prod_{g \in G} B$ by 
\[
(b_g)_{g \in G} \mapsto (\sigma(b_{\sigma^{-1}g}))_{g \in G}, 
\]
then one can check that the diagram 
\[
\begin{CD}
B \otimes_A B @>\phi>> \prod_{g \in G}B \\
@V\sigma\otimes \id_{B}VV @VV\sigma V \\
B \otimes_A B @>\phi>> \prod_{g \in G}B
\end{CD}
\]
is commutative.
Since $\phi(x \otimes 1) = (x)_{g \in G}$, 
if we write $\phi(y_B) = (y_{B,g})_{g \in G}$, 
then 
the equality $(\sigma \otimes \id_B)(y_B) \cdot (x \otimes 1) = y_B$ holds
if and only if 
$\sigma(y_{B, \sigma^{-1}g}) x = y_{B,g}$ for any $g \in G$.
Such a $(y_{B,g})_{g \in G}$ is uniquely determined by $y_{B,1} \in B$. 
In fact, if we set $b = y_{B,1}$, then 
\[
y_{B, \sigma^k} = \sigma^k(b) \prod_{i=0}^{k-1}\sigma^i(x)
\]
for $k > 0$.
Here, we note that $y_{B,\sigma^n} = b$ by the $1$-cocycle condition on $x$.
Hence we deduce that $M_B \cong B$ as $B$-modules. 
\par

Since $f \colon A \rightarrow B$ is faithfully flat, 
we see that $M$ is a locally free $A$-module of rank one
(see \cite[Expos\'e VIII, Corollaire 1.11]{SGA1}).
However, since $\Pic(A) = 0$, 
we conclude that $M$ is a free $A$-module of rank one.
Choose an $A$-basis $y$ of $M$. 
Then $y_B = y \otimes 1$ is a $B$-basis of $M \otimes_A B$. 
If we write $\phi(y_B) = (y_{B,g})_{g \in G}$, 
then $y_{B,1}$ is a $B$-basis of $B$, i.e., $y_{B,1} \in B^\times$.
Since $y_{B, \sigma^k} = \sigma^k(y_{B,1}) \prod_{i=0}^{k-1}\sigma^i(x)$, 
we see that $(y_{B,g})_{g \in G} \in (\prod_{g \in G}B)^\times$, 
hence $y_B = y \otimes 1 \in (B \otimes_A B)^\times$.
\par

Consider a homomorphism $T_y \colon B \rightarrow B$ defined by $T_y(b) = yb$.
Then $T_y \otimes \id_B \colon B \otimes_A B \rightarrow B \otimes_A B$ is an isomorphism.
Since $f \colon A \rightarrow B$ is faithfully flat, 
we deduce that $T_y$ is already an isomorphism. 
Therefore, we conclude that $y \in B^\times$. 
\end{proof}

\begin{rem}
By the same argument, 
one would furthermore generalize Proposition \ref{general90} 
to the case where $G$ is a finite group.
\end{rem}

%\subsection{A lemma}
\subsection{A lemma}
Now we prove the following lemma, which was used in the proof of Lemma \ref{A=>B}.
\begin{lem}\label{lemA}
For any semisimple element $t \in \K_{m,H}$, 
there exists $\tl{s} = s \rtimes \theta \in \tl\K_m$ with $s \in \oo_E[t]$ 
such that $t = s\theta(s)$. 
\end{lem}
\begin{proof}
Write $B = \oo_E[t]$.
Then $B$ is a commutative $\oo_E$-algebra with an involution 
$\sigma(x) = J_N {}^t\overline{x} J_N^{-1}$. 
Note that if $x \in \GL_N(E)$, then $\theta(x) = \sigma(x)^{-1}$. 
Let $A$ be the $\oo_F$-subalgebra of $B$
consisting of $\sigma$-fixed elements in $B$.
Since $A$ is a finitely generated $\oo_F$-module, 
and since $\oo_F$ is a complete local ring, hence henselian, 
we see that $A$ is the direct product of finitely many local rings 
(see \cite[Chapitre I, \S1, D\'efinition 1 and \S2, (3)]{R}). 
Hence $\mathrm{Pic}(A) = 0$ (see e.g., \cite[Proposition 4.3.11]{Wei}). 
\par

Note that the canonical homomorphism $A \otimes_{\oo_F} \oo_E \rightarrow B$ is isomorphism.
Indeed, we have $\epsilon \in \oo_E^\times$ with $\overline{\epsilon} = -\epsilon$ such that 
$\{1,\epsilon\}$ is an $\oo_F$-basis of $\oo_E$. 
Then $\{1,\epsilon\}$ is also a basis of $B$ as an $A$-module. 
Therefore $B$ is a faithfully flat $A$-module. 
Moreover, 
since $E$ is unramified over $F$, 
the homomorphism 
\[
\phi \colon B \otimes_A B \rightarrow B \times B,\; 
b_1 \otimes b_2 \mapsto (b_1b_2, b_1\sigma(b_2))
\]
is isomorphism. 
\par

Therefore, we can apply Proposition \ref{general90}. 
Since $t = \theta(t) = \sigma(t)^{-1}$, 
we can find $s_0 \in B^\times$ such that $t = s_0/\sigma(s_0)$. 
Since $s_0$ and $s_0^{-1}$ are in $B = \oo_E[t]$ with $t \in \K_{m,H}$, both of them are of the form 
\[
\bordermatrix{
&n&1&n \cr
n&\oo_E&\oo_E&\pp_E^{-m}\cr
1&\pp_E^m&\oo_E&\oo_E\cr
n&\pp_E^m&\pp_E^m&\oo_E
}.
\]
Let $u, u' \in \oo_E$ be the $(n+1,n+1)$-entries of $s_0, s_0^{-1}$, respectively. 
Then we have $uu' \equiv u\overline{u'} \equiv 1 \bmod \pp_E^m$ since $t \in \K_{m,H}$.
Hence $u \in \oo_E^\times$, and $u \bmod \pp_E^m$ belongs to $(\oo_F/\pp_F^m)^\times$.
Therefore, we can find $z \in \oo_F^\times$ such that 
$s = z s_0$ satisfies $t = s/\sigma(s) = s\theta(s)$ and $s \in \K_m$.
\end{proof}

%References

\end{document}